\documentclass[11pt]{article}  
\usepackage[]{graphicx}
\usepackage[]{latexsym, amsfonts, amssymb, amstext, enumerate, amsmath, amsthm}
\usepackage{multirow}
\DeclareGraphicsExtensions{ps,eps,pdf}

\theoremstyle{plain}
\newtheorem{theorem}{Theorem}
\newtheorem{corollary}{Corollary}
\newtheorem{proposition}{Proposition}
\newtheorem{lemma}{Lemma}

\theoremstyle{remark}
\newtheorem{remark}{Remark}

\theoremstyle{definition}
\newtheorem{definition}{Definition}

\newtheorem*{EigvEst}{One Step Eigenvector Estimation}{\bf}{\it}

\newcommand{\bJ}{\mathcal{J}}
\newcommand{\Ni}{\mathcal{N}}

\newcommand{\N}{\mathbb{N}}
\newcommand{\be}{{\bf e}}
\newcommand{\bx}{{\bf x}}
\newcommand{\by}{{\bf y}}
\newcommand{\bp}{{\bf p}}
\newcommand{\bz}{{\bf z}}

\newcommand{\R}{\mathbb{R}}

\newcommand{\eigspace}{\mathop{\rm eigenspace}}

\DeclareMathOperator{\dett}{det}
\DeclareMathOperator{\symm}{{\bf Sym}}
\DeclareMathOperator{\diag}{diag}
\DeclareMathOperator{\divg}{div}
\DeclareMathOperator{\argmin}{\mathop{\rm arg\,min}}

\title{An unconstrained framework for eigenvalue problems}
\author{Yunho Kim\footnote{Department of Mathematical Sciences, UNIST, Ulsan, South Korea, Email: yunhokim@unist.ac.kr}}

\begin{document}
\maketitle

\abstract{In this paper, we propose an unconstrained framework for eigenvalue problems in both discrete and continuous settings. We begin our discussion to solve a generalized eigenvalue problem $A{\bx} = \lambda B\bx$ with two $N\times N$ real symmetric matrices $A, B$ via minimizing a proposed functional whose nonzero critical points $\bx\in\R^N$ solve the eigenvalue problem and whose local minimizers are indeed global minimizers. Inspired by the properties of the proposed functional to be minimized, we provide analysis on convergence of various algorithms either to find critical points or local minimizers. Using the same framework, we will also present an eigenvalue problem for differential operators in the continuous setting. It will be interesting to see that this unconstrained framework is designed to find the smallest eigenvalue through matrix addition and multiplication and that a solution $\bx\in\R^N$ and the matrix $B$  can compute the corresponding eigenvalue $\lambda$ without using $A$ in the case of $A\bx=\lambda B\bx$. At the end, we will present a few numerical experiments which will confirm our analysis.
}

\section{Introduction}
\label{intro}

Given an $N\times N$ matrix $A$, the eigenvalue problem of our interest is to find an eigenvalue and its corresponding eigenvector of $A$, that is, to solve $A\bx = \lambda \bx$ for $\bx$ and $\lambda$. This is one of the most fundamental problems in mathematics with applications to all other fields of science. Especially, one may be interested in estimating the largest and the smallest eigenvalues. If $A$ is symmetric and positive definite, then finding the smallest eigenvalue of $A$ is the same as finding the largest eigenvalue of $A^{-1}$, which usually involves solving systems of linear equations of type $A\bx=b$. Then, we ask ourselves a question: ``Is it possible to compute the smallest eigenvalue of $A$ through only basic matrix operations such as multiplication and addition, without solving $A\bx=b$?" 

Our interest of estimating the smallest eigenvalue and its corresponding eigenvector extends to the following infinite dimensional application, as well. On a compact manifold ${\cal M}$, eigenvalues of the Laplacian $-\Delta$  reveals important structures of ${\cal M}$, which makes understanding the eigenvalues of $-\Delta$ on ${\cal M}$ very important. This has interesting applications. For example, in image processing there are a few interesting works (e.g. \cite{N13}, \cite{N12}, \cite{N11}) to distinguish reconstructed objects from point cloud data by evaluating $-\Delta$ on the surfaces of the objects. We can even consider general self-adjoint linear elliptic operators and find their eigenvalues and eigenfunctions. With these theoretical and numerical points of view in mind, our main discussion will be concentrated on finding the smallest eigenvalue and a corresponding eigenvector of a nonzero symmetric matrix, which leads  us to begin with the following well-known constrained problem: given a symmetric and positive definite matrix $A$,

\begin{equation}\label{intro:eigproblem}
\min_{\bx\in\mathbb{R}^N} \langle \bx,A\bx\rangle\quad\text{ subject to }\quad \|\bx\|^2 = 1.
\end{equation}

There have been a large number of works to solve \eqref{intro:eigproblem} by the name of  inverse iteration methods. In particular, we would like to mention the work \cite{N1} by J.E. Dennis and R.A. Tapia, which surveys historical developments of inverse and shifted inverse iterations and of Rayleigh quotient iteration, and which approaches the listed methods from the viewpoint of the Newton's method. 
The unconstrained version of \eqref{intro:eigproblem} analyzed in \cite{N1} is
\begin{equation}\label{intro:N1}
\min_{\bx\in\mathbb{R}^N}\langle \bx,A\bx\rangle + \frac{\gamma}{2}(\|\bx\|^2-1)^2.
\end{equation}
The authors of \cite{N1} explained why the inverse and the shifted inverse Rayleigh quotient iterations are fast and effective by showing the equivalence between the inverse Rayleigh quotient iteration and the Newton's method for \eqref{intro:N1}, and also between the shifted inverse Rayleigh quotient iteration and the Newton's method for the shifted version of \eqref{intro:N1},  when the given matrix $A$ is symmetric and invertible. We refer to \cite{N1}, and references therein, any interested reader in the developments of the inverse and shifted inverse Rayleigh quotient iteration methods. 

There are, however, a few disadvantageous features of the functional in \eqref{intro:N1} that we paid attention to. First of all, \cite{N1} considered only nonsingular matrices for \eqref{intro:N1} just as all other conventional methods do. Second of all, in the simplest case when $A$ is symmetric and positive definite, if $0<\gamma<\lambda_1$, then the zero vector is the only critical point of the functional in \eqref{intro:N1} and, even if $\gamma \geq\lambda_1$, the critical points of the functional in \eqref{intro:N1} are only the eigenvalues of $A$ less than $\gamma$. Hence, our goal is of two folds: 1) extending existing theories to singular matrices,  and 2) removing the additional limitations imposed by the parameter $\gamma$ in \eqref{intro:N1}. Noting that the functional in \eqref{intro:N1} contains the term
$$
(\|\bx\|^2-1)^2 = (\|\bx\|+1)^2(\|\bx\|-1)^2,
$$
we believed that the factor $(\|\bx\|+1)^2$, even though this is convex, is not desirable because the factor tries to push the norm $\|\bx\|$ towards $0$ during minimization.  Therefore, our analysis begins without this factor.

The rest of this manuscript is organized as follows. 
In Section 2, we present our unconstrained framework for solving a generalized eigenvalue problem in a finite dimensional space, where we propose an appropriate functional to be minimized and analyze its interesting properties. We, then, apply the gradient descent method and the Newton's method to the proposed minimization problem for solving eigenvalue problems and provide analysis for convergence either to a global minimizer or to a nonzero critical point.  Moreover, we present a few variants of our approach for quantitative analysis of the error between a true eigenvector and an estimated one. In Section 3, we present the same unconstrained framework for eigenvalue problems in an infinite dimensional space  such as finding eigenfunctions of self-adjoint differential operators, which show universality of our unconstrained framework. In Section 4, we present numerical aspects of our proposed method confirming the theoretical results obtained in the previous sections.

\section{A generalized eigenvalue problem on a finite dimensional space}
\label{sec:1}
First of all, we will consider a generalized eigenvalue problem $A{\bx} = \lambda B\bx$ with two $N\times N$ real symmetric matrices $A,B$, where $B$ is positive definite. Notice that it becomes the usual eigenvalue problem $A\bx = \lambda \bx$ when $B$ is the identity matrix $I$.

Here and in what follows, ${\bf M}_{M\times N}(\R)$, $\symm_N(\R)$, $\symm_{N,p}(\R)$ mean, respectively, the set of $M\times N$ real matrices, the set of $N\times N$ real symmetric matrices, the set of $N\times N$ real symmetric and positive definite matrices. The set of $N\times N$ real matrices will be simply denoted by ${\bf M}_N(\R)$. The case of complex matrices will be mentioned later. Moreover, we consider $\bx\in\R^N$ as an $N\times 1$ column vector and for $\bx, \by\in \R^N$, $\langle \bx,\by\rangle$ will be denoted by $\by^T\bx$ and $\|\bx\|=\sqrt{\langle \bx,\bx\rangle}$. For $A\in{\bf M}_{M\times N}(\R)$, the operator norm of $A$ will be denoted by $\|A\|_{op}$, which is 
$$
\|A\|_{op} = \sup_{\|\bx\|=1}\|A\bx\|.
$$
Note that $\|A\|_{op}$ is the largest singular value of $A$. We also denote by $\sqrt{B}$ for $B\in\symm_{N,p}(\R)$ the matrix $Q\sqrt{\Lambda} Q^T$, where $Q\Lambda Q^T$ is a diagonalization of $B$ and $\sqrt{\Lambda} = \diag(\sqrt{\lambda_1},\dots,\sqrt{\lambda_N})$ when $\Lambda = \diag(\lambda_1,\dots,\lambda_N)$.

Given $A\in \symm_N(\R), B\in \symm_{N,p}(\R)$, we define a functional $F_{A,B}:\R^N\to\R$ by
\begin{equation}\label{eigeq:functional}
F_{A,B}(\bx) = \frac{1}{2}\langle \bx,A\bx\rangle + \frac{\gamma}{2}\langle \bx,B\bx\rangle - \gamma\sqrt{\langle\bx,B\bx\rangle}
\end{equation}
and propose the the following unconstrained problem
\begin{equation}\label{geigpb}
\min_{\bx\in\R^N}F_{A,B}(\bx).
\end{equation}

When $B=I$, we will simply drop the subscript $B$ by writing $F_A$, instead of $F_{A,B}$.
Then, we can see, by a change of variables, $\by = \sqrt{B}\bx$, that \eqref{eigeq:functional} becomes
$$
F_{A,B}(\bx) = \frac{1}{2}\langle \by, C\by\rangle + \frac{\gamma}{2}\|\by\|^2 -\gamma \|\by\| = F_{C}(\by)
$$
with $C =\sqrt{B}^{-1}A\sqrt{B}^{-1}\in\symm_{N}(\R)$, which means that analyzing the functional $F_{A,B}$ is equivalent to analyzing $F_C$. Note that $F_{A,B}$ and $F_C$ are both differentiable at $\bx\neq0$ and that
$$
\nabla F_{A,B}(\bx) =  \nabla F_C(\sqrt{B}\bx)\sqrt{B}
$$
implies the sets of nonzero critical points of $F_{A,B}$ and $F_C$ are equivalent up to the change of variables: $\by = \sqrt{B}\bx$. More precisely,
$$
\nabla F_{A,B}(\bx) = A\bx+\gamma \Big(1- \frac{1}{\sqrt{\langle B\bx,\bx\rangle}}\Big)B\bx\ \text{ and }\ \nabla F_C(\by) = C\by +\gamma\Big(1-\frac{1}{\|\by\|}\Big)\by
$$
imply that we can solve $A\bx = \lambda B\bx$ by finding nonzero critical points of $F_C$.


Therefore, we will begin our discussion with $F_A$ for $A\in\symm_{N}(\R)$ and investigate the minimization problem 
\begin{equation}\label{eigpb}
\min_{\bx\in\R^N}F_A(\bx),
\end{equation}
which is equivalent to
$$
\min_{\bx\in\R^N}\frac{1}{2}\langle \bx,A\bx\rangle +\frac{\gamma}{2}(\|\bx\|-1)^2.
$$

\begin{lemma}\label{eiglem:1}
Let $\lambda_1\in\R$ be the smallest eigenvalue of $A$. For $\gamma>\max(0,-\lambda_1)$, the set of nonzero critical points of $F_A$ is
$$
\Big\{ \bx\in \R^N : A\bx = \lambda \bx  \text{ for some }\lambda\in\R \text{ with } \|\bx\|=\frac{\gamma}{\gamma+ \lambda}\Big\},
$$
and
$$
\min_{\bx\in\R^N} F_A(\bx) = -\frac{\gamma^2}{2(\gamma + \lambda_1)}.
$$
\end{lemma}
\begin{proof}
As was noted above, for $\bx_0\neq0$,
$$
\nabla F_A(\bx_0)= 0\quad \Leftrightarrow\quad A\bx_0 = \gamma\Big(\frac{1}{\|\bx_0\|}-1\Big)\bx_0,
$$
which implies that $\bx_0$ is a nonzero critical point of $F_A$ if and only if $\bx_0$ is an eigenvector of $A$ corresponding to the eigenvalue
$\lambda_0 = \gamma\Big(\frac{1}{\|\bx_0\|}-1\Big)$ with $\|\bx_0\| = \frac{\gamma}{\gamma + \lambda_0}$.
In addition,
$$
F_A(\bx_0) = -\frac{\gamma}{2}\|\bx_0\| = -\frac{\gamma^2}{2(\gamma + \lambda_0)}.
$$
Therefore, the set of nonzero critical points of $F_A$ is
$$
\Big\{ \bx\in \R^N : A\bx=\lambda \bx \text{ for some }\lambda\in\R \text{ with }\|\bx\|= \frac{\gamma}{\gamma+ \lambda}\Big\}.
$$

Due to the choice of $\gamma > \max(0,-\lambda_1)$, it is easy to see that $F_A$ is bounded from below and that a global minimizer $\bx_*$ of $F_A$ exists and is an eigenvector of $A$ corresponding to an eigenvalue $\lambda_*$ with $\|\bx_*\|=\frac{\gamma}{\gamma+ \lambda_*}$.  Since $$F_A(\bx_*) =  -\frac{\gamma^2}{2(\gamma + \lambda_*)} = \min_{\bx\in\R^N}F_A(\bx),$$
we can easily see that  $\lambda_*=\lambda_1$ and
$$
\min_{\bx\in\R^N} F_A(\bx) = -\frac{\gamma^2}{2(\gamma + \lambda_1)}.
$$\qed
\end{proof}

When dealing with $F_A$, we will assume $\gamma > \max(0,-\lambda_1)$, where $\lambda_1$ is the smallest eigenvalue of $A$ if no condition on $\gamma$ is stated. 

\begin{theorem}\label{eigthm:1}
Any local minimizer $\bx_*$ of $F_A$ is a global minimizer.
\end{theorem}

\begin{proof}
First of all, $0$ is not a local minimizer of $F_A$ because for any nonzero $\theta\in\R^N$, we have 
$$
\lim_{t\rightarrow 0^+}\frac{F_A(t\theta)-F_A(0)}{t\|\theta\|} = -\gamma < 0.
$$
Suppose that $\bx_*$ is a local minimizer of $F_A$. Lemma~\ref{eiglem:1} says that $\bx_*$ is an eigenvector of $A$ corresponding to an eigenvalue $\lambda_*$ with $\|\bx_*\|=\frac{\gamma}{\gamma+\lambda_*}$ and that
$$
F_A(\bx_*) = -\frac{\gamma^2}{2(\gamma + \lambda_*)}\geq -\frac{\gamma^2}{2(\gamma + \lambda_1)} = \min_{\bx\in\R} F_A(\bx),
$$
where $\lambda_1$ is the smallest eigenvalue of $A$. We may diagonalize $A$ such that $A = Q\Lambda Q^T$, where $\Lambda$ is a diagonal matrix with nondecreasing diagonal entries $\lambda_1 \leq \cdots \leq \lambda_N$ and $Q$ is an orthogonal matrix having $\frac{\bx_*}{\|\bx_*\|}$ as the $j^{th}$ column for some $j$ implying $\lambda_* = \lambda_j$.  Then, it suffices to show that $\lambda_j = \lambda_1$, which implies that $\bx_*$ is a global minimizer.

Suppose that $\lambda_j>\lambda_1$. With $\by=Q^T\bx$, we have
\begin{equation}\label{equiv:diagonal}
F_A(\bx) =  \ F_A(Q\by) = F_\Lambda(\by).
\end{equation}
For $k=1,2,\dots, N$, we set ${\bf e}_k$ to be the $k^{th}$ column of the identity matrix $I\in {\bf M}_N(\R).$ Then, we can see that $\bx_*$ being a local minimizer of $F_A$ is equivalent to $\frac{\gamma}{\gamma+\lambda_j}{\bf e}_j$ being a local minimizer of $F_\Lambda$ and that
\begin{equation}\label{equiv:diagonalF}
F_\Lambda(\by) = \frac{1}{2} \langle \by, \Lambda \by\rangle +\frac{\gamma}{2}\|\by\|^2 - \gamma\|\by\| = \Big(\frac{1}{2}\sum_{k=1}^N \Big(\lambda_k + \gamma\Big)y_k^2\Big) -\gamma\|\by\|,
\end{equation}
where $\by = [y_1\ \cdots\ y_N]^T$. Moreover, $\frac{\gamma}{\gamma + \lambda_1}{\bf e}_1$ is a global minimizer of $F_\Lambda$. Since $\lambda_j>\lambda_1$, $\be_j$ is orthogonal to $\be_1$ and  we can consider $F_\Lambda$ on the subspace spanned by $\{{\bf e}_1, {\bf e}_j\}$ by defining $H:\R^2\rightarrow \R$ by
$$
H(a,b) = F_\Lambda(a{\bf e}_1 + b{\bf e}_j) = \frac{1}{2}\Big(\lambda_1 + \gamma\Big)a^2 + \frac{1}{2}\Big(\lambda_j + \gamma\Big)b^2-\gamma\sqrt{a^2 + b^2}.
$$
Then, $\nabla^2H(0,\frac{\gamma}{\gamma + \lambda_j})$ exists and must be positive semidefinite, i.e., $$\dett\Big[\nabla^2H\Big(0,\frac{\gamma}{\gamma + \lambda_j}\Big)\Big]\geq0.$$ However, at $(a,b) = (0,\frac{\gamma}{\gamma+\lambda_1})$, we obtain that
\begin{align*}
\dett(\nabla^2H(a,b)) =&\ \dett\left[\begin{array}{cc}(\gamma+\lambda_1) - \frac{\gamma b^2}{\sqrt{a^2 + b^2}^3} & \frac{\gamma ab}{\sqrt{a^2 + b^2}^3}\\ \frac{\gamma ab}{\sqrt{a^2 + b^2}^3}&(\gamma+\lambda_j) - \frac{\gamma a^2}{\sqrt{a^2 + b^2}^3}\end{array}\right]\\
 =&\ (\lambda_1- \lambda_j)(\gamma+\lambda_j)<0,
\end{align*}
which is a contradiction. Therefore, $\lambda_*=\lambda_j = \lambda_1$, i.e., any local minimizer $\bx_*$ of $F_A$ is a global minimizer. \qed
\end{proof}

In addition, we may be able to find all the eigenvalues and their corresponding eigenvectors of $A$.

\begin{corollary}\label{eigthm:2}
Let $\lambda_1\leq \cdots\leq \lambda_N$ be the eigenvalues of $A$. Let $\{\bx_1, \dots, \bx_k\}$ be an orthonormal set of eigenvectors of $A$ corresponding to the eigenvalues $\lambda_1,\dots, \lambda_k$. We consider the following problem
\begin{equation}\label{eigeq:functionalsub}
\min_{\bx\in\R^N} F_A(\bx)\ \text{ subject to }\ \langle \bx,{\bf x}_i\rangle = 0,\ i=1,2,\dots,k.
\end{equation}
Then, any local minimizer $\bx_*$ of  \eqref{eigeq:functionalsub} is a global minimizer corresponding to the eigenvalue $\lambda_{k+1}$ with $\|\bx_*\|=\frac{\gamma}{\gamma+\lambda_{k+1}}$.
\end{corollary}

\begin{proof}
With a diagonalization $Q\Lambda Q^T$ of $A$, where  $\Lambda=\operatorname{diag}(\lambda_1,\lambda_2,\dots,\lambda_N)$ and the first $k$ columns of $Q$ are ${\bf x}_1,\dots,{\bf x}_k$, and $\by = Q^T\bx$, we have that $F_A(\bx) = F_\Lambda(\by)$ and $\|\bx\| = \|\by\|$, so \eqref{eigeq:functionalsub} is equivalent to
\begin{equation}\label{eigeq:small}
\min \{F_\Lambda(\by) : \by = [0\ \cdots\ 0\ y_{k+1}\ \cdots\ y_N]^T\in \R^N\}.
\end{equation}
Let $\Lambda_k$ be the last $(n-k)\times (n-k)$ block of $\Lambda$, i.e., $\Lambda_k=\operatorname{diag}(\lambda_{k+1},\dots,\lambda_N)$. Then, \eqref{eigeq:small} is equivalent to
$$
\min_{\bz\in\R^{N-k}}F_{\Lambda_k}(\bz).
$$
Theorem~\ref{eigthm:1} applies to $F_{\Lambda_k}$ and we are done.\qed
\end{proof}


Even though $F_A$ is not convex, we have seen that all local minimizers of $F_A$ are global minimizers, which is a rare case for nonconvex functionals, and that nonzero critical points of $F_A$ are eigenvectors of $A$. Hence, one can expect that any algorithm either to minimize the functional $F_A$ or to find a critical point of $F_A$ will work. For example, if we set $G(\bx) = \|\bx\|$, then the conjugate functional $G^*$ is
$$
G^*(\by) = \sup_{\bx\in\R^N}\langle \bx,\by\rangle -G(\bx) = \begin{cases}0, & \text{ if } \|\by\|\leq 1,\\ \infty, & \text{ if } \|\by\|>1,\end{cases}
$$
and $G(\bx) = G^{**}(\bx) = \sup_{\by\in\R^N}\langle \bx,\by\rangle -G^*(\by) =  \sup_{\|\by\|\leq1}\langle \bx,\by\rangle$. Then, \eqref{eigpb} becomes
$$
\min_{\substack{\bx\in\R^N,\\ \|\by\|\leq 1}} \frac{1}{2}\langle \bx,A\bx\rangle +\frac{\gamma}{2}\|\bx\|^2 - \gamma\langle \bx,\by\rangle.
$$
In fact, the constraint on $\by$ is $\|\by\|=1$, so we have
\begin{equation}\label{eigpb:constrained}
\min_{\substack{\bx\in\R^N,\\ \|\by\| = 1}} F_A(\bx,\by) =\frac{1}{2}\langle \bx,A\bx\rangle +\frac{\gamma}{2}\|\bx\|^2 - \gamma\langle \bx,\by\rangle.
\end{equation}
Since the functional $F_A(\cdot, \by)$ is convex and quadratic in $\bx$ for any fixed $\by$, we may consider an algorithm such as
\begin{enumerate}
\item $\tilde{\bx}=\argmin_{\bx\in\R^N} F_A(\bx,\by)$,
\item Update $\by$ to be $\frac{\tilde{\bx}}{\|\tilde{\bx}\|}$.
\item Iterate the above procedure until it converges.
\end{enumerate}
This algorithm is exactly the inverse power method
$$
\Big(\frac{1}{\gamma}A + I\Big)\bx_{k+1} = \frac{\bx_k}{\|\bx_k\|}.
$$
However, \eqref{eigpb:constrained} is a constrained problem and the above algorithm requires solving a system of linear equation at every iteration, not to mention, the rate of convergence is linear. Therefore, we want to consider an algorithm that satisfies either one of the following two:
\begin{itemize}
\item the rate of convergence is linear, yet the algorithm applies only matrix addition and multiplication,
\item the rate of convergence is faster than linear if we need to solve systems of linear equations.
\end{itemize}

\subsection{The gradient descent method}\label{subsec:1}

As for the first algorithm, we will analyze the gradient descent method for the minimization problem \eqref{eigpb} with stepsize $\alpha_k>0$: with $\bx_0\neq 0$,
\begin{equation}\label{min_seq}
\bx_{k+1} = \bx_k - \alpha_k\nabla F_A(\bx_k) = \bx_k - \alpha_k\Big(A\bx_k + \gamma \bx_k - \frac{\gamma}{\|\bx_k\|}\bx_k\Big).
\end{equation}

Let $\{q_1,\dots,q_N\}$ be an orthonormal basis for $\R^N$ consisting of unit eigenvectors of $A$ corresponding to the eigenvalues $\lambda_1\leq \cdots\leq \lambda_N$, respectively. If $\bx_k = \mu_{k,1}q_1 + \cdots + \mu_{k,N} q_N$, then
\begin{align}\label{min_seq1}
\bx_{k+1} =&\ \mu_{k+1,1}q_1 + \cdots + \mu_{k+1,N} q_N = \sum_{i=1}^N\mu_{k,i}\Big[1-\alpha_k\Big(\lambda_i + \gamma -\frac{\gamma}{\|\bx_k\|}\Big)\Big]q_i\nonumber\\
=&\ \sum_{i=1}^N\mu_{0,i}\Pi_{j=0}^k\Big[1-\alpha_j\Big(\lambda_i + \gamma -\frac{\gamma}{\|x_j\|}\Big)\Big]q_i.
\end{align}
For simplicity, we assume a fixed stepsize $0< \alpha_k = \alpha < \frac{1}{\lambda_N+\gamma}$ for all $k\in \N$. Then, \eqref{min_seq} always converges to a critical point of $F_A$. In fact, it converges to a global minimizer with probability 1 if an initial point $\bx_0\neq 0$ is chosen randomly. Theorem~\ref{thm:convergence} below is given in a general form.

\begin{theorem}\label{thm:convergence}
A sequence $\{\bx_k\}$ generated by \eqref{min_seq} with $\bx_0 = \mu_{0,1}q_1 + \cdots + \mu_{0,N}q_N \neq 0$ converges to a critical point $\bx_*$ of $F_A$ which is an eigenvector of $A$ corresponding the eigenvalue $\lambda_l$ with 
$$
\|\bx_*\| = \frac{\gamma}{\gamma + \lambda_l},
$$
where $l = \min\{j\in\{1,\dots,N\} : \mu_{0,j}\neq0\}$. More precisely, $\bx_*$ is
$$
\Big(\frac{\gamma}{\gamma+\lambda_l}\Big)\Big(\frac{1}{\sqrt{\sum_{m: \lambda_m=\lambda_l} \mu_{0,m}^2}}\Big)\sum_{m: \lambda_m=\lambda_l} \mu_{0,m}q_m.
$$
\end{theorem}
\begin{proof}
For any $\bx\neq 0$, we get
\begin{equation}\label{nablaFnorm}
\Big\langle \nabla F_A(\bx),\frac{\bx}{\|\bx\|}\Big\rangle = \Big(\frac{\langle A\bx,\bx\rangle}{\|\bx\|^2} + \gamma\Big)\|\bx\| - \gamma = (\lambda+\gamma)\|\bx\|-\gamma,
\end{equation}
where $\lambda_1\leq \lambda \leq \lambda_N.$ Then, $\alpha < \frac{1}{\lambda_N+\gamma}$ implies $\|\bx_k\|> \alpha\gamma$ for all $k\in\N$ since
$$
\|\bx_{k+1}\|^2 \geq \Big|\Big\langle \bx_k - \alpha\nabla F_A(\bx_k),\frac{\bx_k}{\|\bx_k\|}\Big\rangle\Big|^2\geq \bigl(\|\bx_k\|-\alpha(\lambda+\gamma)\|\bx_k\|+\alpha\gamma\bigr)^2 > (\alpha\gamma)^2.
$$
Moreover, the line segment connecting $\bx_k$ and $\bx_{k+1}$ for $k\in\N$ lies entirely in $\{\bx \in \R^N : \|\bx\| \geq \alpha\gamma\}$. To see this, we take $0<t<1$ and observe that for $k\in\N$,
\begin{align*}
\|\bx_k-t\alpha\nabla F(\bx_k)\|^2 \geq&\ \bigl(\|\bx_k\|-t\alpha(\lambda+\gamma)\|\bx_k\|+t\alpha\gamma\bigr)^2\\
>&\ (\alpha\gamma)^2(1-t\alpha(\lambda+\gamma) + t)^2 > (\alpha\gamma)^2.
\end{align*}
Noting that
$$
\nabla^2F_A(\bx) = (A+\gamma I) -\frac{\gamma}{\|\bx\|}\Big(I-\frac{\bx\bx^T}{\|\bx\|^2}\Big)
$$
and $A+\gamma I,\ I-\frac{\bx\bx^T}{\|\bx\|^2}$ are positive semidefinite with $\|A+\gamma I\|_{op} \leq \lambda_N+\gamma$ and $\|I-\frac{\bx\bx^T}{\|\bx\|^2}\|_{op}\leq1$, we can see that  for $\|\bx\|\geq \alpha\gamma$, $$\|\nabla^2F_A(\bx)\|_{op} \leq \max(\lambda_N + \gamma, \frac{1}{\alpha}) \leq \frac{1}{\alpha},$$ 
from which we obtain that for each $k\geq 1$, 
\begin{align*}
F_A(\bx_{k+1}) \leq &\ F_A(\bx_k) + \langle \nabla F_A(\bx_k), \bx_{k+1}-\bx_k\rangle + \frac{1}{2\alpha}\|\bx_{k+1}-\bx_k\|^2\\
=&\ F_A(\bx_k) -\frac{\alpha}{2}\|\nabla F_A(\bx_k)\|^2,
\end{align*}
which implies
$$
\|\nabla F_A(\bx_k)\|^2 \leq \frac{2}{\alpha}(F_A(\bx_k) - F_A(\bx_{k+1})).
$$
Note that for any $K\geq 1$,
\begin{equation}\label{eig:convergence}
\frac{1}{\alpha^2}\sum_{k=1}^K\|\bx_{k+1}-\bx_k\|^2 = \sum_{k=1}^K\|\nabla F_A(\bx_k)\|^2 \leq \frac{2}{\alpha}(F_A(\bx_1) - (\min_{\bx\in\R^N}F_A(\bx))).
\end{equation}
Since $F_A$ is coercive and $F_A(\bx_k) \leq F_A(\bx_1)<\infty$ for all $k\geq1$, $\{\bx_k\}$ must be a bounded sequence in $$\{\bx \in \R^N : \|\bx\| \geq \alpha\gamma\}.$$ 

Choosing a convergent subsequence $\{\bx_{k_n}\}$ to $\bx_*$, we know from \eqref{eig:convergence}  that $\nabla F_A(\bx_*) = 0$, i.e., $\bx_*$ is an eigenvector of $A$ corresponding to an eigenvalue $\lambda_i$ for some $1\leq i\leq N$ with norm
$
\|\bx_*\| = \frac{\gamma}{\gamma + \lambda_i}.
$
Knowing that $\{F_A(\bx_k)\}$ is a decreasing and bounded sequence, we can easily derive that any subsequential limit $\tilde{\bx}$ of $\{\bx_k\}$ satisfies $F_A(\tilde{\bx}) = F_A(\bx_*)$ and $A\tilde{\bx} = \lambda_i\tilde{\bx}$ with norm
$
\|\tilde{\bx}\| = \frac{\gamma}{\gamma + \lambda_i}.
$
Hence, we can conclude that
$$
\lim_{k\rightarrow\infty}\|\bx_k\| =\frac{\gamma}{\gamma+\lambda_i}.
$$
Setting $l = \min\{j\in \{1,\dots,N\} : \mu_{0,j}\neq 0\}$, we can see from \eqref{min_seq1} that $\lambda_i\geq \lambda_l$. Suppose $\lambda_i>\lambda_l$. 
We note that for all $k\in\N$,
$$
1-\alpha\Big(\lambda_l +\gamma -\frac{\gamma}{\|\bx_k\|}\Big)  >1-\frac{\lambda_l+\gamma}{\lambda_N+\gamma}+\frac{\alpha\gamma}{\|\bx_k\|} \geq\frac{\alpha\gamma}{\|\bx_k\|}>0
$$
and that as $k\rightarrow\infty$,
$$
1-\alpha\Big(\lambda_l +\gamma -\frac{\gamma}{\|\bx_k\|}\Big)\rightarrow 1-\alpha(\lambda_l-\lambda_i)>1.
$$
From \eqref{min_seq1}, we can see that $$\mu_{0,l}\Pi_{j=0}^k\Big[1-\alpha\Big(\lambda_l + \gamma -\frac{\gamma}{\|\bx_j\|}\Big)\Big]\rightarrow \infty\ \text{ as }\ k\rightarrow\infty.$$ This is a contradiction because $\{\bx_k\}$ is a bounded sequence. Therefore, $\lambda_i=\lambda_l$.

Moreover, for $\lambda_p>\lambda_l$, we have
$$
1-\alpha\Big(\lambda_p +\gamma -\frac{\gamma}{\|\bx_k\|}\Big)\rightarrow 1-\alpha(\lambda_p-\lambda_l)\in(0,1) \ \text{ as }\ k\rightarrow\infty,
$$
implying
$$
\mu_{0,p}\Pi_{j=0}^k\Big[1-\alpha\Big(\lambda_p+ \gamma -\frac{\gamma}{\|\bx_j\|}\Big)\Big]\rightarrow 0\ \text{ as }\ k\rightarrow\infty.
$$

Hence, referring to \eqref{min_seq1}, we can see that
$$
\bx_{k+1} - \Big(\sum_{m: \lambda_m=\lambda_l}\mu_{0,m}q_m\Big)\Pi_{j=0}^k\Big[1-\alpha\Big(\lambda_l + \gamma -\frac{\gamma}{\|\bx_j\|}\Big)\Big] \rightarrow 0\ \text{ as }\ k\rightarrow\infty.
$$
In addition, convergence of the norm $\|\bx_k\|$ implies that
$$
\Pi_{j=0}^\infty\Big[1-\alpha\Big(\lambda_l + \gamma -\frac{\gamma}{\|\bx_j\|}\Big)\Big] = \frac{\gamma}{\gamma+\lambda_l}\Big(\frac{1}{\sqrt{\sum_{m: \lambda_m=\lambda_l}\mu_{0,m}^2}}\Big).
$$
Therefore,
$\bx_k$ converges to
$$
\frac{\gamma}{\gamma+\lambda_l}\Big(\frac{1}{\sqrt{\sum_{m: \lambda_m=\lambda_l}\mu_{0,m}^2}}\Big)\Big(\sum_{m: \lambda_m=\lambda_l}\mu_{0,m}q_m\Big).
$$
\qed
\end{proof}

Now, going back to the generalized eigenvalue problem
\begin{equation}\label{gep}
A\bx = \lambda B\bx,
\end{equation} 
via minimizing $F_{A,B}$ in \eqref{geigpb}, we realize that even though \eqref{geigpb} is equivalent to \eqref{eigpb}, the gradient descent method we discussed above is applicable to $F_C$ with $C=\sqrt{B}^{-1}A\sqrt{B}^{-1}$ to find a critical point of $F_{A,B}$. That means that not only do we need to compute $\sqrt{B}$, but also we need to invert it. However, it turns out that applying the gradient descent method directly to $F_{A,B}$ to solve \eqref{geigpb} finds solutions of \eqref{gep}, as well. 

When solving \eqref{geigpb}, we will use $\mu_{(A,j)}, \mu_{(B,j)}$ to denote the $j^{th}$ smallest eigenvalues of $A$ and $B$, respectively. Moreover, we assume that $\mu_{(B,1)} = 1$ and that either $\mu_{(A,1)}\geq0$ or $-\mu_{(A,N)}\leq \mu_{(A,1)} <0$ is true. Note that these assumptions are not restrictions, but simplifications. In relation to \eqref{gep}, when dealing with $F_{A,B}$, the parameter $\gamma$ will be assumed to satisfy $\gamma >\max(0,-\mu_{(A,1)})$.

We set $\{r_1,\dots,r_N\}$ to be an orthonormal set of eigenvectors of $C$ corresponding to the eigenvalues $\lambda_1\leq \cdots\leq \lambda_N$. We also set $q_j = \sqrt{B}^{-1}r_j$ for $1\leq j\leq N$. Then, it is easy to see that $\{q_1,\dots,q_N\}$ is an orthonormal basis for $\R^N$ with respect to the inner product $\langle \bx,\by\rangle_B:= \bx^TB\by$. 

\begin{theorem}\label{thm:gdgeigpb}
If we choose
$$
0<\alpha < \frac{1}{(\mu_{(A,N)} + \gamma)(\mu_{(B,N)})^3},
$$
then with $\bx_0\neq 0$, the following procedure
\begin{equation}\label{gep:gdconv}
\bx_{k+1} = \bx_k -\alpha\nabla F_{A,B}(\bx_k),\ k=0,1,2,\dots,
\end{equation}
produces a sequence $\{\bx_k\}$ converging to $\bx_*$ with $\sqrt{\langle B\bx_*,\bx_*\rangle} =\frac{\gamma}{\gamma+\lambda_*}$, where $(\bx_*,\lambda_*)$ is a solution pair of \eqref{gep}.
\end{theorem}

\begin{proof}
If $B=I$, then this theorem is the same as Theorem~\ref{thm:convergence}. Hence, we will assume that $\mu_{(B,N)}>1$. In addition, since the proof mimics that of Theorem~\ref{thm:convergence} with minor differences in detail, we will emphasize only those minor differences. Let $\{\bx_k\}$ be the sequence generated by \eqref{gep:gdconv}. First of all, we note that for any $\bx_k\in\R^N$, 
\begin{equation}\label{gdgeigpb:eq1}
\Big|\frac{\langle A\bx_k,B\bx_k\rangle}{\langle B\bx_k,\bx_k\rangle}\Big|\leq \mu_{(A,N)}\mu_{(B,N)}\ \text{and}\ 1=\mu_{(B,1)}\leq \frac{\langle B\bx_k,B\bx_k\rangle}{\langle B\bx_k,\bx_k\rangle}\leq \mu_{(B,N)}.
\end{equation}
This implies that if $0<\alpha<\frac{1}{(\mu_{(A,N)} + \gamma)(\mu_{(B,N)})^3}$, then for $0<\epsilon :=1-\frac{1}{(\mu_{(B,N)})^2}$, 
\begin{equation}\label{gdgeigpb:eq2}
1-\alpha\Big[\frac{\langle A\bx_k,B\bx_k\rangle}{\langle B\bx_k,\bx_k\rangle} + \gamma\frac{\langle B\bx_k,B\bx_k\rangle}{\langle B\bx_k,\bx_k\rangle}\Big]\geq 1-\alpha(\mu_{(A,N)} + \gamma)\mu_{(B,N)} >\epsilon.
\end{equation}
We can see by \eqref{gdgeigpb:eq1} and \eqref{gdgeigpb:eq2} that for $k=0,1,2,\dots,$ 
\begin{align}\label{gd_conv:ineq}
\langle B\bx_{k+1},\bx_{k+1}\rangle \geq &\ \frac{\langle B\bx_{k+1},\bx_k\rangle^2}{\langle B\bx_k,\bx_k\rangle}=\Big(\frac{\langle \bx_k-\alpha\nabla F_{A,B}(\bx_k), B\bx_k\rangle}{\sqrt{\langle B\bx_k,\bx_k\rangle}}\Big)^2\nonumber\\
\geq&\ (\sqrt{\langle B\bx_k,\bx_k\rangle}\epsilon + \alpha\gamma)^2 \nonumber\\
=&\ \langle B\bx_k,\bx_k\rangle\epsilon^2 + 2\sqrt{\langle B\bx_k,\bx_k\rangle}\alpha\gamma\epsilon + (\alpha\gamma)^2.
\end{align}
This proves that $\langle B\bx_{k},\bx_{k}\rangle \geq (\alpha\gamma)^2$ for $k\geq 1$, which can be improved further as follows: for $k\geq 1$, we have
$$
\langle B\bx_k,\bx_k\rangle\epsilon^2 + 2\sqrt{\langle B\bx_k,\bx_k\rangle}\alpha\gamma\epsilon + (\alpha\gamma)^2 \geq \langle B\bx_k,\bx_k\rangle\epsilon^2 + (\alpha\gamma)^2(1+2\epsilon),
$$
resulting in
$$
\langle B\bx_{k+1},\bx_{k+1}\rangle \geq (\alpha\gamma)^2(1+2\epsilon)\sum_{l=0}^{k-1}\epsilon^{2l} + (\alpha\gamma)^2\epsilon^{2k}.
$$
Hence,
$$
\liminf_{k\rightarrow\infty}\langle B\bx_k,\bx_k\rangle \geq (\alpha\gamma)^2\frac{1+2\epsilon}{1-\epsilon^2}> \frac{(\alpha\gamma)^2}{1-\epsilon} = (\alpha\gamma\mu_{(B,N)})^2
$$
and there exists $K\in\N$ such that $k\geq K$ implies
$$
\langle B\bx_k,\bx_k\rangle> (\alpha\gamma\mu_{(B,N)})^2.
$$

Moreover, we can estimate $\langle B(\bx_k +t(\bx_{k+1}-\bx_k)),(\bx_k +t(\bx_{k+1}-\bx_k))\rangle$  for $t\in(0,1)$ as follows:
 for $k>1,2,\dots,$ and $0<t<1$,
\begin{align*}
&\ \langle B(\bx_k +t(\bx_{k+1}-\bx_k)),(\bx_k +t(\bx_{k+1}-\bx_k))\rangle = \|\sqrt{B}(\bx_k -t\alpha\nabla F_A(\bx_k))\|^2\\
\geq&\ \Big\langle \sqrt{B}(\bx_k-t\alpha\nabla F_{A,B}(\bx_k)),\frac{\sqrt{B}\bx_k}{\|\sqrt{B}\bx_k\|}\Big\rangle^2=\Big(\frac{\langle \bx_k-t\alpha\nabla F_{A,B}(\bx_k), B\bx_k\rangle}{\sqrt{\langle B\bx_k,\bx_k\rangle}}\Big)^2\\
=&\ \Big[\sqrt{\langle B\bx_k,\bx_k\rangle}\Big(1-t\alpha\Big[\frac{\langle A\bx_k,B\bx_k\rangle}{\langle B\bx_k,\bx_k\rangle} + \gamma\frac{\langle B\bx_k,B\bx_k\rangle}{\langle B\bx_k,\bx_k\rangle}\Big]\Big) + t\alpha\gamma\frac{\langle B\bx_k,B\bx_k\rangle}{\langle B\bx_k,\bx_k\rangle}\Big]^2\\
\geq&\ (\sqrt{\langle B\bx_k,\bx_k\rangle}(1-t\epsilon_1) + t\alpha\gamma)^2 =(\sqrt{\langle B\bx_k,\bx_k\rangle} +t(\alpha\gamma-\sqrt{\langle B\bx_k,\bx_k\rangle}\epsilon_1))^2\\
\geq&\ \begin{cases} \langle B\bx_k,\bx_k\rangle, &\text{ if } \alpha\gamma \geq \sqrt{\langle B\bx_k,\bx_k\rangle}\epsilon_1,\\
(\sqrt{\langle B\bx_k,\bx_k\rangle}(1-\epsilon_1) + \alpha\gamma)^2, &\text{ if }\alpha\gamma < \sqrt{\langle B\bx_k,\bx_k\rangle}\epsilon_1. \end{cases}
\end{align*}
where $\epsilon_1 = 1-\epsilon$. Noting that
$\alpha\gamma < \sqrt{\langle B\bx_k,\bx_k\rangle}\epsilon_1$ implies
$$
\sqrt{\langle B\bx_k,\bx_k\rangle}(1-\epsilon_1) + \alpha\gamma> \frac{\alpha\gamma(1-\epsilon_1)}{\epsilon_1} + \alpha\gamma = \frac{\alpha\gamma}{1-\epsilon} >\alpha\gamma\mu_{(B,N)},
$$
we have that for $k\geq K$, 
$$
\min_{t\in[0,1]}\langle  B(\bx_k +t(\bx_{k+1}-\bx_k)),(\bx_k +t(\bx_{k+1}-\bx_k))\rangle > (\alpha\gamma\mu_{(B,N)})^2.
$$
Therefore, we have
$$
\{\bx_k+ t(\bx_{k+1}-\bx_k) : t\in[0,1], k\geq K\} \subset \{\bx\in\R^N : \langle B\bx,\bx\rangle \geq (\alpha\gamma\mu_{(B,N)})^2\}.
$$ 

Since $\sqrt{B}\Big(\frac{\gamma}{\sqrt{\langle B\bx,\bx\rangle}}\Big[I-\Big(\frac{\sqrt{B}\bx}{\|\sqrt{B}\bx\|}\Big)\Big(\frac{\sqrt{B}\bx}{\|\sqrt{B}\bx\|}\Big)^T\Big]\Big)\sqrt{B}$ is positive semidefinite and
$$
\Big\|\sqrt{B}\Big(\frac{\gamma}{\sqrt{\langle B\bx,\bx\rangle}}\Big[I-\Big(\frac{\sqrt{B}\bx}{\|\sqrt{B}\bx\|}\Big)\Big(\frac{\sqrt{B}\bx}{\|\sqrt{B}\bx\|}\Big)^T\Big]\Big)\sqrt{B},\Big\|_{op}\leq \frac{1}{\alpha},
$$
and $\|A+\gamma B\|_{op}\leq \frac{1}{\alpha}$ and
$$
\nabla^2 F_{A,B}(\bx) = A+ \gamma B -\sqrt{B}\Big(\frac{\gamma}{\sqrt{\langle B\bx,\bx\rangle}}\Big[I-\Big(\frac{\sqrt{B}\bx}{\|\sqrt{B}\bx\|}\Big)\Big(\frac{\sqrt{B}\bx}{\|\sqrt{B}\bx\|}\Big)^T\Big]\Big)\sqrt{B},
$$
we can see that $\|\nabla^2F_{A,B}(\bx)\|_{op}\leq \frac{1}{\alpha}$.

The rest of the proof for convergence of the sequence $\{\bx_k\}_{k>K}$ to $\bx_*$ satisfying
$$
\sqrt{\langle B\bx_*,\bx_*\rangle} = \|\sqrt{B}\bx_*\|= \frac{\gamma}{\gamma + \lambda_*},
$$
where $(\bx_*,\lambda_*)$ is a solution pair of \eqref{gep}, is omitted due to the similarity of that in Theorem~\ref{thm:convergence}.
\qed\end{proof}

In Theorem~\ref{thm:gdgeigpb} above, convergence to a nonzero critical point of $F_{A,B}$ is confirmed. However, if we can efficiently deal with $B^{-1}$, then we can guarantee to find a global minimizer of $F_{A,B}$ by considering the gradient descent method with respect to a different inner product.

\begin{corollary}\label{thm:gdgeigpb1}
If we generate a sequence $\{\bx_k\}$ by
\begin{equation}\label{gd_gep1}
\bx_{k+1} = \bx_k - \alpha B^{-1}\nabla F_{A,B}(\bx_k),\ k=0,1,2\dots,
\end{equation}
with a randomly chosen $\bx_0\neq 0$ and $0<\alpha<\frac{1}{\mu_{(A,N)}+\gamma}$, then the sequence converges to $\bx_*$, where $(\bx_*,\lambda_*)$ is a solution pair of \eqref{gep} satisfying
$$
\sqrt{\langle B\bx_*,\bx_*\rangle} = \|\sqrt{B}\bx_*\| = \frac{\gamma}{\gamma+\lambda_*}
$$
and
$$
\lambda_* = \min\{\lambda : A-\lambda B \text{ is singular}\}.
$$
In fact, \eqref{gd_gep1} is the gradient descent of $F_{A,B}$ with respect to the inner product $\langle \bx,\by\rangle_B := \bx^TB\by.$
\end{corollary}

\begin{proof}
First of all, as we mentioned, \eqref{gep} is equivalent to
$$
B^{-1}A\bx = \lambda \bx,
$$
and to $$
C\by =\lambda \by,\ \text{ with } C =\sqrt{B}^{-1}A\sqrt{B}^{-1},\  \by=\sqrt{B}\bx.
$$
Therefore, $(\bx,\lambda)$ is a solution pair of \eqref{gep} if and only if $\bx$ is an eigenvector of $B^{-1}A$ corresponding to the eigenvalue $\lambda$ if and only if $\sqrt{B}\bx$ is an eigenvector of $C$ corresponding to the eigenvalue $\lambda$. Note that since $\langle \by, C\by\rangle = \frac{\langle \bx,A\bx\rangle}{\|\bx\|^2}\|\bx\|^2$ for $\bx = \sqrt{B}^{-1}\by$, and $\min_{\|\by\|=1}\|\sqrt{B}^{-1}\by\|\leq \frac{1}{\sqrt{\mu_{(B,1)}}}=1$,
$$
\lambda_1 =\min_{\|\by\|=1}\langle \by, C\by\rangle \geq \min_{\|\by\|=1}\mu_{(A,1)}\|\sqrt{B}^{-1}\by\|^2\geq \begin{cases}\mu_{(A,1)},&\text{ if } \mu_{(A,1)}<0,\\ \ 0,&\text{ if }\mu_{(A,1)}\geq 0.,\end{cases}
$$
Since $\gamma>\max(0,-\mu_{(A,1)})$ implies $\gamma >\max(0,-\lambda_1)$, we can apply Theorem~\ref{thm:convergence} to $$F_C(\by)=\frac{1}{2}\langle\by,C\by\rangle +\frac{\gamma}{2}\|\by\|^2 -\gamma\|\gamma\|.$$ Moreover, since $\lambda_N$, the largest eigenvalue of $C$, is at most $\mu_{(A,N)}>0$, we know that $0<\alpha<\frac{1}{\mu_{(A,N)}+\gamma}\leq \frac{1}{\lambda_N+\gamma}$ and that a sequence $\{\by_k\}$ generated by
$$
\by_{k+1} = \by_k -\alpha\nabla F_C(\by_k),\ k=0,1,2,\dots,
$$
with $\by_0\neq0$ chosen at random, converges to $\by_*$, an eigenvector of $C$ corresponding to the smallest eigenvalue $\lambda_*=\lambda_1$ with norm $\|\by_*\| = \frac{\gamma}{\gamma+\lambda_1}$.

On the other hand, with $\by_k = \sqrt{B}\bx_k$, $k=0,1,\dots$, we have
$$
\nabla F_C(\by_k) =(\sqrt{B})^{-1}\nabla F_{A,B}(\bx_k),
$$
which implies that \eqref{gd_gep1} is nothing but
$$
\by_{k+1} = \by_k -\alpha\nabla F_C(\by_k).
$$

Hence, the sequence generated by \eqref{gd_gep1}, with a randomly chosen $\bx_0\neq0$, converges to $\bx_*$, where $(\bx_*,\lambda_*)$ is a solution pair of \eqref{gep} satisfying 
$$\sqrt{\langle B\bx_*,\bx_*\rangle}= \frac{\gamma}{\gamma+\lambda_*}\ \text{ and }\ \lambda_* = \lambda_1 = \min\{\lambda : A-\lambda B \text{ is singular}\}.
$$

In addition, with respect to the inner product $\langle \bx,\by\rangle_B:=\bx^TB\by$, we note that 
$$
B^{-1}\nabla F_{A,B}(\bx) = B^{-1}A\bx + \gamma\Big(1-\frac{1}{\sqrt{\langle \bx,\bx\rangle_B}}\Big)\bx,
$$
which is the gradient of $F_{A,B}$ with respect to the inner product $\langle\cdot,\cdot\rangle_B$ because
$$
F_{A,B}(\bx) = \frac{1}{2}\langle \bx, B^{-1}A\bx\rangle_B + \frac{\gamma}{2}\|\bx\|_B^2 -\gamma\|\bx\|_B,
$$
Note also that $B^{-1}A$ is self-adjoint with respect to $\langle\cdot,\cdot\rangle_B$. Therefore, \eqref{gd_gep1} is the gradient descent of $F_{A,B}$ with respect to $\langle\cdot,\cdot\rangle_B$.
\qed\end{proof}

In general, with a nonsymmetrix square matrix $E$, minimizing the functional $F_E$ does not guarantee to find an eigenvector of $E$.  However, Theorem~\ref{thm:gdgeigpb} and Corollary~\ref{thm:gdgeigpb1} make it possible to find eigenvectors of $E$ in certain cases when $E$ decomposes into $E=B^{-1}A$ with a symmetric matrix $A$ and a symmetric positive definite matrix $B$. Moreover, if it is easy to compute $B^{-1}$, then Corollary~\ref{thm:gdgeigpb1} applies to find a global minimizer of $F_E$ with a better stepsize. We can also find subsequent eigenvectors in the same way as presented in Theorem~\ref{eigthm:2}.

\begin{corollary}\label{prop:multiple}
Let $(\vec{x}_1,\lambda_1), \dots, (\vec{x}_m,\lambda_m)$ be solution pairs of \eqref{gep} where $\lambda_1\leq \dots \leq \lambda_m$ are the first $m\geq 1$ smallest ones in $\{\lambda\in \R : A-\lambda B\ \text{ is singular}\}.$

We consider the following problem
\begin{equation}\label{eigeq:functionalsub1}
\min_{\bx\in\R^N} F_{A,B}(\bx)\ \text{ subject to }\ \langle \bx,\bx_k\rangle_B = 0,\ k=1,2,\dots,m.
\end{equation}
Then, any local minimizer $\bx_*$ of $F_{A,B}$ in the subspace orthogonal to $\{\bx_1,\dots,\bx_m\}$ with respect to the inner product $\langle\cdot,\cdot\rangle_B$, is a solution with $\sqrt{\langle B\bx_*,\bx_*\rangle} =\frac{\gamma}{\gamma+\lambda_*}$ and
$$
\lambda_* = \min(\{\lambda\in\R : A-\lambda B\ \text{ is singular}\}\setminus\{\lambda_1,\dots,\lambda_m\}).
$$
\end{corollary}

\subsection{The Newton's method}

As for the second algorithm with a faster rate of convergence, we will analyze the Newton's method to find nonzero critical points of \eqref{eigpb}, which are eigenvectors of $A$.
Since the functional $F_A$ is continuously twice differentiable at $\bx\neq 0$, if we apply the Newton's method, we will generate a sequence $\{\bx_k\}$ by
\begin{equation}\label{newton}
\bx_{k+1} = \bx_k -(\nabla^2F_A(\bx_k))^{-1}\nabla F_A(\bx_k),\ k=0,1,\dots,
\end{equation}
with an initial $\bx_0\neq 0$ unless $\nabla^2F(\bx_k)$ is singular.  We can observe that  \eqref{newton} becomes for $k=0,1,2,\dots,$
$$
\bx_{k+1} = \Big[\frac{1}{\gamma}A + \Big(1-\frac{1}{\|\bx_k\|}\Big)I + \frac{1}{\|\bx_k\|}\Big(\frac{\bx_k}{\|\bx_k\|}\Big)\Big(\frac{\bx_k}{\|\bx_k\|}\Big)^T\Big]^{-1}\frac{\bx_k}{\|\bx_k\|}.
$$
Hence, we propose the following scheme: with an initial guess $\bx_0\neq 0$,
for $k=0,1,2,\dots,$ compute $\by_k$ and $\bx_{k+1}$ by
$
\by_k = \frac{\bx_k}{\|\bx_k\|},
$
and
\begin{equation}\label{newton1}
\Big[\frac{1}{\gamma}A + \Big(1-\frac{1}{\|\bx_k\|}\Big)I + \frac{1}{\|\bx_k\|}\Big(\frac{\bx_k}{\|\bx_k\|}\Big)\Big(\frac{\bx_k}{\|\bx_k\|}\Big)^T\Big]\bx_{k+1} = \by_k.
\end{equation}

We wrote \eqref{newton1} in the given form to enhance its similarity either to the inverse iteration or to the Rayleigh quotient iteration. 

As for convergence, we will show that convergence of the norm $\|\bx_k\|$ is equivalent to convergence of $\bx_k$.
\begin{theorem}\label{newton:convergence}
Let $A\in\symm_N(\R)$ have eigenvalues $\lambda_1\leq \cdots \leq \lambda_N$. Let $\gamma>\max(0,-\lambda_1)$, and $0<\|\bx_0\|\neq \frac{\gamma}{\gamma+\lambda_j}$ for any $1\leq j\leq N$. 

\begin{enumerate}
\item Suppose that a sequence $\{\bx_k\}_{k=1}^\infty$ can be generated by \eqref{newton1}, i.e., $\bx_k$ is computable for all $k\in\N$, and that $\|\bx_k\|$ converges to $\eta>0$.
If $\|\bx_k\|\neq \frac{\gamma}{\gamma+\lambda_j}$ for any $1\leq j\leq N$ and for all $k\geq1$, then there exists $1\leq i_0\leq N$ such that $\eta = \frac{\gamma}{\gamma+\lambda_{i_0}}$ and $\bx_k$ converges to an eigenvector $\bx_*$ of $A$ corresponding to the eigenvalue $\lambda_{i_0}$ with $\|\bx_*\| = \frac{\gamma}{\gamma+\lambda_{i_0}}$. 
\item On the other hand, suppose that we generate a sequence $\{\bx_k\}_{k=1}^{k_0}$ for some $k_0\in\N$ by \eqref{newton1} and $\bx_{k_0}$ satisfies $\|\bx_{k_0}\| = \frac{\gamma}{\gamma+\lambda_{i}}$, where $\lambda_i$ for some $1\leq i\leq N$ is an eigenvalue of $A$ with multiplicity 1. Let  $q_i$ be a unit eigenvector of $A$ corresponding to $\lambda_i$. If $\bx_{k_0}$ is not a critical point of $F_A$ with $|q_i^T\bx_{k_0}|>0$, then $\bx_{k_0+1}$ is an eigenvector of $A$ corresponding to the eigenvalue $\lambda_i$. If $|q_i^T\bx_{k_0}|\neq \frac{\gamma(\gamma+\lambda_j)}{(\gamma+\lambda_i)^2}$ for $j< i$, then $\bx_{k_0+2}$ is a critical point of $F_A$, i.e., an eigenvector of $A$ corresponding to the eigenvalue $\lambda_i$ with norm $\|\bx_{k_0+2}\| = \frac{\gamma}{\gamma+\lambda_i}$. However, if $|q_i^T\bx_{k_0}|= \frac{\gamma(\gamma+\lambda_j)}{(\gamma+\lambda_i)^2}$ for some $j< i$, then the system becomes singular and we may not compute $\bx_{k_0+2}$ uniquely. In any case, the algorithm terminates in $k_0+2$ iterations.
\end{enumerate}
\end{theorem}

\begin{proof}
It suffices to consider the case that $A$ is a diagonal matrix with diagonal entries $\lambda_1\leq\lambda_2\leq\cdots\leq\lambda_N$. Then, \eqref{newton1} becomes
\begin{equation}\label{newton2}
\Big(1+\frac{\lambda_j}{\gamma}-\frac{1}{\|\bx_k\|}\Big) \bx_{k+1,j} = \by_{k,j}\Big(1-\frac{\|\bx_{k+1}\|}{\|\bx_k\|}(\by_k^T\by_{k+1})\Big),\ j=1,2,\dots,N,
\end{equation}
where $\bx_{k+1} = \begin{bmatrix}\bx_{k+1,1}& \cdots & \bx_{k+1,N}\end{bmatrix}^T$ and $\by_k = \frac{\bx_k}{\|\bx_k\|}$.  
Let $\{\bx_k\}$ be a sequence generated by \eqref{newton2} with $\|\bx_k\|\neq\frac{\gamma}{\gamma+\lambda_j}$ for any $1\leq j\leq N$ and for all $k\in\N\cup\{0\}$.

Firstly, we consider the case that $\|\bx_k\|$ converges to $\eta>0$. Since $\bx_k$ is computable for all $k\in\N$, we can see from \eqref{newton2} that $$1 - \frac{\|\bx_{k+1}\|}{\|\bx_{k}\|}(\by_k^T\by_{k+1})\neq 0,\quad k\geq 0.$$
By setting $\bJ_{0}:=\{j\in\{1,2,\dots,N\} : \bx_{0,j}\neq 0\}$, we know that for $k\geq0$,
$$
\bx_{k,j} \neq 0 \text{ if and only if } j\in \bJ_{0}.
$$
We will now prove $\limsup_k \by_k^T\by_{k+1}=1$ by contradiction. Suppose that $$\limsup_k \by_k^T\by_{k+1}<1.$$ Then, there exists $\epsilon<1$ with $\limsup_k \by_k^T\by_{k+1}=\epsilon$. Given $\delta\in(0,1-\epsilon)$, we may choose $l_1\in\N$ so that $k\geq l_1$ implies 
\begin{equation}\label{newton:condition}
|\|\bx_k\|-\eta|<\frac{\delta}{2}\quad\text{and}\quad\Big|\frac{\|\bx_{k}\|}{\|\bx_{k+1}\|} -1\Big|<\frac{\delta}{2}\quad\text{and}\quad \by_k^T\by_{k+1}<\epsilon+\frac{\delta}{2}.
\end{equation}
We also choose $J\in\bJ_{0}$ satisfying
$$
\Big|\eta\Big(1+\frac{\lambda_J}{\gamma}\Big)-1\Big| = \min_{j\in \bJ_{0}}\Big|\eta\Big(1+\frac{\lambda_j}{\gamma}\Big)-1\Big|.
$$ 
From \eqref{newton2}, we see that for $k\geq l_1$,
\begin{equation}\label{newton4}
\frac{|\bx_{k+1,J}|}{\|\bx_{k+1}\|} \geq \frac{1-\delta-\epsilon}{\Big|\eta(1+\frac{\lambda_J}{\gamma})-1\Big| + \frac{\delta}{2}(1+\frac{\lambda_J}{\gamma})}\frac{|\bx_{k,J}|}{\|\bx_{k}\|}.
\end{equation}
If $\Big|\eta\Big(1+\frac{\lambda_J}{\gamma}\Big)-1\Big| <1-\epsilon$, then by choosing $\delta\in(0,1-\epsilon)$ satisfying
$$
\Big|\eta\Big(1+\frac{\lambda_J}{\gamma}\Big)-1\Big| +\frac{\delta}{2}\Big(1+\frac{\lambda_J}{\gamma}\Big) < 1-\delta-\epsilon,
$$
we can see from \eqref{newton4} that $\lim_{k\rightarrow\infty}\frac{|\bx_{k+1,J}|}{\|\bx_{k+1}\|} =\infty$, which is impossible. Hence,
$$
\min_{j\in \bJ_{0}}\Big|\eta\Big(1+\frac{\lambda_j}{\gamma}\Big)-1\Big| \geq 1-\epsilon
\quad\text{ i.e., }\quad
\eta \leq \frac{\epsilon\gamma}{\gamma + \lambda^*}\  \text{ or }\  \eta \geq \frac{(2-\epsilon)\gamma}{\gamma  +\lambda_*},
$$
where $\lambda_* =\min_{j\in \bJ_{0}} \lambda_j$ and $\lambda^* =\max_{j\in \bJ_{0}} \lambda_j$.

Suppose that $\eta \leq \frac{\epsilon\gamma}{\gamma + \lambda^*}$.
For any $0<\delta<\frac{(1-\epsilon)\gamma}{2\gamma+\lambda^*}<1-\epsilon$, we can see from \eqref{newton:condition} that for each $j\in \bJ_{0}$, $k\geq l_1$ implies 
$$
\|\bx_k\|\Big(1+\frac{\lambda_j}{\gamma}\Big)-1 < \Big(\eta+\frac{\delta}{2}\Big)\Big(1+\frac{\lambda_j}{\gamma}\Big)-1\leq \frac{\epsilon(\gamma+\lambda_j)}{\gamma+\lambda^*} + \frac{\delta}{2}\Big(1+\frac{\lambda_j}{\gamma}\Big)-1<0,
$$
and 
$$
\frac{\|\bx_k\|}{\|\bx_{k+1}\|} - (\by_k^T\by_{k+1}) > 1-\frac{\delta}{2}-(\by_k^T\by_{k+1})>0.
$$
This results in, for $k\geq l_1$,
\begin{align*}
\by_k^T\by_{k+1}=&\  \Big(\frac{\|\bx_k\|}{\|\bx_{k+1}\|} - (\by_k^T\by_{k+1})\Big)\Big(\sum_{j=1}^N\frac{\by_{k,j}^2}{\|\bx_k\|(1+\frac{\lambda_j}{\gamma})-1}\Big)\\
\leq&\ \Big(1-\frac{\delta}{2} - (\by_k^T\by_{k+1})\Big)\Big(\sum_{j\in\bJ_0}\frac{\by_{k,j}^2}{\|\bx_k\|(1+\frac{\lambda_j}{\gamma})-1}\Big)<0.
\end{align*}
This implies that $\epsilon \leq 0$, i.e., $\eta\leq 0$, which is a contradiction, i.e., $\eta \leq \frac{\epsilon\gamma}{\gamma + \lambda^*}$ is not possible. Hence, we must have $\eta\geq  \frac{(2-\epsilon)\gamma}{\gamma  +\lambda_*}$.
Again with $0<\delta<\frac{(1-\epsilon)\gamma}{2\gamma+\lambda_*}<1-\epsilon$,  we can also see that for $k\geq l_1$, and for each $j\in \bJ_{0}$, 
\begin{align*}
\|\bx_k\|\Big(1+\frac{\lambda_j}{\gamma}\Big)-1 >&\ \Big(\eta-\frac{\delta}{2}\Big)\Big(1+\frac{\lambda_j}{\gamma}\Big)-1\\
\geq&\ \frac{(2-\epsilon)(\gamma+\lambda_j)}{\gamma+\lambda_*} - \frac{\delta}{2}\Big(1+\frac{\lambda_j}{\gamma}\Big)-1\\
>&\ \frac{(2-\epsilon)(\gamma+\lambda_j)}{\gamma+\lambda_*} - \frac{(1-\epsilon)(\gamma+\lambda_j)}{2(2\gamma+\lambda_*)}-1>0,
\end{align*}
which implies that for $k\geq l_1$, since $\epsilon+\delta<1$,
\begin{align*}
\by_k^T\by_{k+1}=&\  \Big(\frac{\|\bx_k\|}{\|\bx_{k+1}\|} - (\by_k^T\by_{k+1})\Big)\Big(\sum_{j=1}^N\frac{\by_{k,j}^2}{\|\bx_k\|(1+\frac{\lambda_j}{\gamma})-1}\Big)\\
\geq&\ \Big(1-\frac{\delta}{2} - (\by_k^T\by_{k+1})\Big)\Big(\sum_{j\in\bJ_0}\frac{\by_{k,j}^2}{\|\bx_k\|(1+\frac{\lambda_j}{\gamma})-1}\Big)>0.
\end{align*}
Hence, we have
\begin{equation}\label{newton:epsilon}
0\leq \epsilon<1.
\end{equation}

If we extract a subsequence  $\by_{k_n}$ such that $\by_{k_n}^T\by_{{k_n}+1}\rightarrow \epsilon$ as $n\rightarrow\infty$, then using the form of \eqref{newton1} and knowing that 
$$
\liminf_{n\rightarrow\infty}\frac{\by_{{k_n}+1}^T}{\|\bx_{{k_n}+1}\|}\frac{1}{\gamma}A\bx_{{k_n}+1} =\liminf_{n\rightarrow\infty} \frac{1}{\gamma}\by_{{k_n}+1}^TA\by_{{k_n}+1}\geq \frac{\lambda_*}{\gamma} 
$$
 we can see that
\begin{align*}
\frac{\by_{{k_n}+1}^T}{\|\bx_{{k_n}+1}\|}&\Big[\frac{1}{\gamma}A + \Big(1-\frac{1}{\|\bx_{k_n}\|}\Big)I + \frac{1}{\|\bx_{k_n}\|}\Big(\frac{\bx_{k_n}}{\|\bx_{k_n}\|}\Big)\Big(\frac{\bx_{k_n}}{\|\bx_{k_n}\|}\Big)^T\Big]\bx_{{k_n}+1} = \frac{\by_{{k_n}+1}^T\by_{k_n}}{\|\bx_{{k_n}+1}\|}\\
\Leftrightarrow &\quad \frac{1}{\gamma}\by_{{k_n}+1}^TA\by_{{k_n}+1} + \Big(1-\frac{1}{\|\bx_{k_n}\|}\Big) +\frac{1}{\|\bx_{k_n}\|}(\by_{k_n}^T\by_{k_n+1})^2 = \frac{\by_{{k_n}+1}^T\by_{k_n}}{\|\bx_{{k_n}+1}\|}\\
\Rightarrow &\quad \frac{\lambda_*}{\gamma} + \Big(1 - \frac{1}{\eta}\Big) + \frac{\epsilon^2}{\eta} \leq \frac{\epsilon}{\eta}\\
\Leftrightarrow &\quad \eta \leq \frac{(1+\epsilon-\epsilon^2)\gamma}{\gamma+\lambda_*}.
\end{align*}
Therefore, we obtain
\begin{equation}\label{newton:conclusion1}
 \frac{(2-\epsilon)\gamma}{\gamma  +\lambda_*} \leq\eta\leq  \frac{(1+\epsilon-\epsilon^2)\gamma}{\gamma  +\lambda_*}.
\end{equation}
However, this is a contradiction since \eqref{newton:epsilon} implies $1+\epsilon-\epsilon^2 < 2-\epsilon$, i.e., $\eta\geq \frac{(2-\epsilon)\gamma}{\gamma+\lambda_*}$ is not possible, either. Therefore, we conclude that $\epsilon<1$ is impossible, i.e., $\limsup_{k} \by_k^T\by_{k+1} = 1$.

We can now show that $\eta = \frac{\gamma}{\gamma+\lambda_{i_0}}$ for some $1\leq i_0\leq N$. Suppose that $\eta\neq \frac{\gamma}{\gamma  +\lambda_j}$ for any $1\leq j\leq N$. By considering a subsequence $\{\by_{k_n}\}$ with $\lim_{n\rightarrow \infty}\by_{k_n}^T\by_{{k_n}+1}=1$, it is easy to see using \eqref{newton2} that 
\begin{align*}
\infty =&\ \lim_{n\rightarrow\infty}\frac{1}{(\frac{\|\bx_{k_n}\|}{\|\bx_{{k_n}+1}\|} -(\by_{k_n}^T\by_{{k_n}+1}))^2} = \lim_{n\rightarrow\infty}\sum_{j=1}^N \frac{\by_{{k_n},j}^2}{(\|\bx_{k_n}\|(1+\frac{\lambda_j}{\gamma})-1)^2}\\
 \leq &\ \max_{1\leq j\leq N}\frac{1}{(\eta(1+\frac{\lambda_j}{\gamma})-1)^2}<\infty,
\end{align*}
which is a contradiction. Hence, we have that $\eta = \frac{\gamma}{\gamma+\lambda_{i_0}}$ for some $1\leq i_0\leq N$.

Next, we will show that $\bx_k$ converges to an eigenvector $\bx_*$ of $A$ corresponding to $\lambda_{i_0}$ with norm $\|\bx_*\|=\eta=\frac{\gamma}{\gamma+\lambda_{i_0}}$. Firstly, we show  that there exists $j\in\bJ_0$ such that $\lambda_j=\lambda_{i_0}$. As above, by choosing a subsequence $\by_{k_n}$ such that $\by_{k_n}^T\by_{k_n+1}\rightarrow 1$ as $n\rightarrow\infty$, it is easy to see that there must be $j\in\bJ_0$ with $\lambda_j = \lambda_{i_0}$. That is, $\{j\in\bJ_0 : \lambda_j=\lambda_{i_0}\}\neq\emptyset$. Hence, without loss of generality we will say that $i_0\in\bJ_0$. 

Let $k_0\in\N$ be such that $k\geq k_0$ implies
$$
\Big|1+\frac{\lambda_{i_0}}{\gamma} - \frac{1}{\|\bx_k\|}\Big|<\min_{\lambda_j\neq \lambda_{i_0}}\frac{|\lambda_j-\lambda_{i_0}|}{3\gamma}.
$$
Then, for $k\geq k_0$,  and for $\lambda_j\neq \lambda_{i_0}$,
$$
\Big|\frac{1+\frac{\lambda_{i_0}}{\gamma} - \frac{1}{\|\bx_k\|}}{1+\frac{\lambda_{j}}{\gamma} - \frac{1}{\|\bx_k\|}}\Big|<\frac{1}{2}.
$$
Moreover,
since we have that for $1\leq j\leq N$, 
\begin{equation}\label{newton:infprod}
\frac{\bx_{K+1,j}}{\|\bx_{K+1}\|} = \Big(\Pi_{k=0}^K \Big[\frac{\frac{1}{\|\bx_{k+1}\|} - \frac{1}{\|\bx_{k}\|}(\by_k^T\by_{k+1})}{1+\frac{\lambda_j}{\gamma}-\frac{1}{\|\bx_k\|}}\Big]\Big)\frac{\bx_{0,j}}{\|\bx_{0}\|},\  \ K\geq 0,
\end{equation}
if we choose $j$ for which $\lambda_j\neq \lambda_{i_0}$, we can see that
\begin{align*}
\Big|\frac{\bx_{K+1,j}}{\bx_{K+1,i_0}}\Big| =&\ \Big(\Pi_{k=k_0}^K \Big|\frac{1+\frac{\lambda_{i_0}}{\gamma}-\frac{1}{\|\bx_k\|}}{1+\frac{\lambda_j}{\gamma}-\frac{1}{\|\bx_k\|}}\Big|\Big)\Big(\Pi_{k=0}^{k_0-1} \Big|\frac{1+\frac{\lambda_{i_0}}{\gamma}-\frac{1}{\|\bx_k\|}}{1+\frac{\lambda_j}{\gamma}-\frac{1}{\|\bx_k\|}}\Big|\Big)\Big|\frac{\bx_{0,j}}{\bx_{0,i_0}}\Big|\\
\leq&\ \frac{1}{2^{K-k_0+1}}\Big(\Pi_{k=0}^{k_0-1} \Big|\frac{1+\frac{\lambda_{i_0}}{\gamma}-\frac{1}{\|\bx_k\|}}{1+\frac{\lambda_j}{\gamma}-\frac{1}{\|\bx_k\|}}\Big|\Big)\Big|\frac{\bx_{0,j}}{\bx_{0,i_0}}\Big|\rightarrow 0\ \text{ as }\ K\rightarrow \infty.
\end{align*}
Since  $\frac{\bx_{K+1,j}}{\bx_{K+1,i_0}} = \frac{\bx_{0,j}}{\bx_{0,i_0}}$ for $\lambda_j = \lambda_{i_0}$, we can also see that $\|\bx_{K+1}\|^2\rightarrow \eta^2$ as $K\to\infty$ implies 
$$
(\bx_{K+1,i_0})^2\sum_{\lambda_j=\lambda_{i_0}} \Big(\frac{\bx_{0,j}}{\bx_{0,i_0}}\Big)^2\rightarrow \Big(\frac{\gamma}{\gamma+\lambda_{i_0}}\Big)^2\ \text{ as }\ K\to\infty.
$$
Let $m_{i_0} =\Big(\sum_{\lambda_j=\lambda_{i_0}} (\frac{\bx_{0,j}}{\bx_{0,i_0}})^2\Big)^{\frac{1}{2}}\geq 1$.
Then, as $K\rightarrow\infty$,
\begin{equation}\label{newton:convx}
(\bx_{K+1,i_0})^2\rightarrow \frac{1}{m_{i_0}^2}\Big(\frac{\gamma}{\gamma+\lambda_{i_0}}\Big)^2.
\end{equation}
This implies that $(\by_{K+1,i_0})^2\rightarrow\frac{1}{m_{i_0}^2}$ as $K\rightarrow\infty$. 

In addition, noting that for all $k$,
$$
1\geq |\by_k^T\by_{k+1}|=  \Big|\frac{\|\bx_k\|}{\|\bx_{k+1}\|} - (\by_k^T\by_{k+1})\Big|\Big|\sum_{j=1}^N\frac{\by_{k,j}^2}{\|\bx_k\|(1+\frac{\lambda_j}{\gamma})-1}\Big|,
$$
and $\lim_{k\to\infty}\Big|\sum_{j=1}^N\frac{\by_{k,j}^2}{\|\bx_k\|(1+\frac{\lambda_j}{\gamma})-1}\Big| = \infty$,
we know that 
$$
\lim_{k\rightarrow\infty}\Big|\frac{\|\bx_k\|}{\|\bx_{k+1}\|} - (\by_k^T\by_{k+1})\Big| = 0.
$$
Hence, not only do we have $\limsup_{k\to\infty}\by_k^T\by_{k+1} = 1$, but also we can obtain that $$\lim_{k\rightarrow\infty}\by_k^T\by_{k+1} = 1.$$
In fact, since $\bx_{k,j} = \bx_{k,i_0}(\frac{\bx_{0,j}}{\bx_{0,i_0}})$ for $\lambda_j=\lambda_{i_0}$, we can see that 
$$
\by_k^T\by_{k+1} = \frac{1}{\|\bx_k\|\|\bx_{k+1}\|}\Big(\sum_{\lambda_j=\lambda_{i_0}}\bx_{k,i_0}\bx_{k+1,i_0}\Big(\frac{\bx_{0,j}}{\bx_{0,i_0}}\Big)^2+\sum_{\lambda_j\neq\lambda_{i_0}}\bx_{k,j}\bx_{k+1,j}\Big),
$$
and that $\lim_{k\rightarrow\infty}\by_k^T\by_{k+1} = 1$ implies
$$
\lim_{k\rightarrow\infty}(\bx_{k,i_0}\bx_{k+1,i_0}) = \frac{1}{m_{i_0}^2}\Big(\frac{\gamma}{\gamma+\lambda_{i_0}}\Big)^2.
$$
Together with \eqref{newton:convx}, we know that $\lim_{k\to\infty}\bx_{k,i_0}$ exists and is either $\frac{1}{m_{i_0}}\frac{\gamma}{\gamma+\lambda_{i_0}}$ or $-\frac{1}{m_{i_0}}\frac{\gamma}{\gamma+\lambda_{i_0}}$. Letting $\bx_{*,i_0} = \lim_{k\to\infty}\bx_{k,i_0}$, we have that $\bx_k$ converges to $\bx_*$, where
$$
\bx_{*,j} = \begin{cases}(\frac{\bx_{0,j}}{\bx_{0,i_0}})\bx_{*,i_0}, &\text{ for } \lambda_j=\lambda_{i_0},\\ \qquad 0,& \text{ for } \lambda_j\neq\lambda_{i_0}.\end{cases}
$$
Note that $\bx_*$ is an eigenvector of $A$ corresponding to $\lambda_{i_0}$ with norm $\|\bx_*\| = \eta = \frac{\gamma}{\gamma+\lambda_{i_0}}$. This finishes the first part of the theorem.

For the second part of the theorem, we generate a sequence $\{\bx_k\}_{k=1}^{k_0}$ for some $k_0\in\N$ and suppose that $\bx_{k_0}$ satisfies $\|\bx_{k_0}\|  =\frac{\gamma}{\gamma+\lambda_i}$, where $\lambda_i$ for some $1\leq i\leq N$ is an eigenvalue of $A$ with multiplicity $1$. If $\bx_{k_0}$ is not a critical point of $F_A$ and $\bx_{k_0,i}\neq 0$, then we have $|\bx_{k_0,i}|<\frac{\gamma}{\gamma+\lambda_i}$ and $\frac{1}{\|\bx_{k_0}\|}(\by_{k_0}^T\bx_{k_0+1}) =1$, which turns \eqref{newton2} into
$$
\frac{1}{\gamma}\begin{bmatrix}\lambda_1 & 0 & \cdots & 0\\ 0 & \lambda_2 & \cdots & 0\\ \vdots & \vdots & \ddots & \vdots\\ 0 & 0 & \cdots & \lambda_N\end{bmatrix}\begin{bmatrix}\bx_{k_0+1,1}\\ \vdots\\ \bx_{k_0+1,N}\end{bmatrix} =\frac{\lambda_i}{\gamma}\begin{bmatrix}\bx_{k_0+1,1}\\ \vdots\\ \bx_{k_0+1,N}\end{bmatrix}.
$$
Since $\lambda_i$ is of multiplicity 1, there exists a unique solution $\bx_{k_0+1} = \alpha \vec{e}_i$, where $\vec{e}_i$ is the standard basis element in $\R^N$ with $\vec{e}_{i,j} = \delta_{ij}$ and $\alpha = \frac{\gamma^2}{\bx_{k_0,i}(\gamma+\lambda_i)^2}$. Note that $|\alpha|>\frac{\gamma}{\gamma+\lambda_i}$ and $\bx_{k_0+1}$ is an eigenvector of $A$ corresponding to the eigenvalue $\lambda_i$, and yet is  not a critical point of $F_A$.

Since $\bx_{k_0+2}$ satisfies \eqref{newton2} with $\by_{k_0+1}=\pm\vec{e}_i$, we can see that for $1\leq j\leq N$,
\begin{equation}\label{newton_system}
\begin{cases}(\frac{\lambda_j}{\gamma} + 1 -\frac{1}{|\alpha|})\bx_{k_0+2,j} = \delta_{ij}(1-\frac{1}{|\alpha|}\bx_{k_0+2,i}), &\text{ if } \by_{k_0+1} = \vec{e}_i,\\
(\frac{\lambda_j}{\gamma} + 1 -\frac{1}{|\alpha|})\bx_{k_0+2,j}  = -\delta_{ij}(1+\frac{1}{|\alpha|}\bx_{k_0+2,i}), &\text{ if } \by_{k_0+1} = -\vec{e}_i.
\end{cases}
\end{equation}
Note that  $|\alpha| = \frac{\gamma}{\gamma+\lambda_j}$ is equivalent to $|\bx_{k_0,i}| = \frac{\gamma(\gamma+\lambda_j)}{(\gamma+\lambda_i)^2}$. Since $|\bx_{k_0,i}|<\frac{\gamma}{\gamma+\lambda_i}$, it is possible to have $|\alpha| = \frac{\gamma}{\gamma+\lambda_j}$ only if $j<i$.

Hence, if $|\bx_{k_0,i}|\neq \frac{\gamma(\gamma+\lambda_j)}{(\gamma+\lambda_i)^2}$ for $j< i$, then $|\alpha|\neq \frac{\gamma}{\gamma+\lambda_j}$ for $j<i$, and  \eqref{newton_system} is nonsingular and has a unique solution 
 $$
 \bx_{k_0+2,j} = \begin{cases}\pm\frac{\gamma}{\gamma+\lambda_i}, & \text{ if } j=i,\\ \quad 0, & \text{ if } j\neq i.\end{cases}
$$
depending on $\by_{k_0+1} = \pm\vec{e}_i$. That is, $\bx_{k_0+2}$ is a critical point of $F_A$, an eigenvector of $A$ corresponding to the eigenvalue $\lambda_i$ with norm $\|\bx_{k_0+2}\| = \frac{\gamma}{\gamma+\lambda_i}$ and the algorithm terminates.

On the other hand, if $|\bx_{k_0,i}|= \frac{\gamma(\gamma+\lambda_j)}{(\gamma+\lambda_i)^2}$ for some $j< i$, then $|\alpha| = \frac{\gamma}{\gamma+\lambda_j}$ and the system \eqref{newton_system} becomes singular and the algorithm terminates.\qed
\end{proof}

\begin{remark}
It is very interesting to note that in both of the gradient descent method and the Newton's method discussed above, the convergence of a generated sequence $\{\bx_k\}$ is confirmed by  the convergence of the sequence of their norms $\{\|\bx_k\|\}$, which hardly happens, in general. So, it would probably be worth further investigation in a subsequent work.
\end{remark}

\subsection{Some variants of the proposed framework}

We have seen in the previous sections how to solve the generalized eigenvalue problem \eqref{gep} $$A\bx =\lambda B\bx$$ with $A\in\symm_N(\R),\ B\in\symm_{N,p}(\R)$ in an unconstrained framework. In this section, we will proceed our discussion on some variants,  inspired by \eqref{newton1}, of our proposed framework including nonsymmetric cases, as well.

When $\bx_k$ converges to an eigenvalue $\bx_*$ of $A$ corresponding to an eigenvalue $\lambda_*$ via \eqref{newton1}, we know that $\gamma(1-\frac{1}{\|\bx_k\|})$ converges to $-\lambda_*$, hence \eqref{newton1} becomes
$$
(A-\lambda_*I + (\gamma+\lambda_*)\tilde{\by}\tilde{\by}^T)\bx = \gamma\tilde{\by}
$$
with $\tilde{\by} = \frac{\bx_*}{\|\bx_*\|}$. Hence, we may consider the following procedure.
\begin{EigvEst} Given $A\in\symm_N(\R)$, and an eigenvalue $\tilde{\lambda}$ of $A$, and $\gamma>0$ with $\gamma\neq- \tilde{\lambda}$, we choose $\bx_0$ uniformly at random from $S^{N-1}$ and solve for $\bx$,
\begin{equation}\label{newton:onestep}
(A-\tilde{\lambda}I + (\gamma+\tilde{\lambda})\bx_0\bx_0^T)\bx = \gamma \bx_0.
\end{equation}
\end{EigvEst}

In the case of $0$ being an eigenvalue of $A$, by setting $\tilde{\lambda}=0$, \eqref{newton:onestep} turns into $(A + \gamma \bx_0\bx_0^T)\bx = \gamma \bx_0$, which means that  \eqref{newton:onestep} is equivalent to
$$
(A-\tilde{\lambda}I + \gamma \bx_0\bx_0^T) \bx = \gamma \bx_0, \ \gamma\neq 0.
$$

Before proceeding our discussion, we want to mention a work \cite{N14} of G. Peters and J.H. Wilkinson, which was further explained in \cite{N1}. In \cite{N14}, the authors discussed an idea of computing an approximate eigenvector $x_\lambda$ when an approximate eigenvalue $\lambda$ is given, i.e., when $A-\lambda I$ is very ill-conditioned, or near singular, by considering
\begin{equation}\label{pw}
(A-\lambda I + \bx_i\bp^T)\bx = \bx_i,
\end{equation}
with a random vector $\bp$, inspired by the inverse iteration, i.e., by $(A-\lambda I)\bx_{i+1} = \frac{\bx_i}{\|\bx_i\|}$. The authors noticed that $A-\lambda I +\bx_i\bp^T$ can be well-conditioned and the solution to \eqref{pw} is nothing but a constant multiple of the solution to $(A-\lambda I)\bx=\bx_i$, but provided reasons why \eqref{pw} is not in their favor simply because $A-\lambda I + \bx_i\bp^T$ changes its form at every iteration making computations inefficient. However, with $\lambda$ fixed, even though a limit exists for the inverse iteration, the convergence is still linear.  Further discussions can be found in \cite{N15}, and the references therein, in relation to the (shifted) inverse iteration and the Rayleigh quotient iteration.

On the other hand, our concern is if we can make full use of the nonsingular system \eqref{newton:onestep}  to analyze quantitatively the error in eigenvector estimation regardless of the multiplicities of the corresponding eigenvalues helping understand the Newton's method \eqref{newton}. So, we will provide a series of results for the rest of this section.

\begin{proposition}\label{prop:newton1}
Suppose that $\tilde{\lambda}$ has multiplicity 1. With probability 1, the equation \eqref{newton:onestep} has a unique nonzero solution $\tilde{\bx}$ that is an eigenvector of $A$ corresponding to the eigenvalue $\tilde{\lambda}$.
\end{proposition}
\begin{proof}
Let $q$ be a unit eigenvector of $A$ corresponding to $\tilde{\lambda}$. Note that if we choose $\bx_0\in S^{N-1}$ uniformly at random, then we have $q^T\bx_{0}\neq 0$ with probability 1. Moreover, if $q^T\bx_0\neq 0$, then $(A -\tilde{\lambda}I + (\gamma+\tilde{\lambda})(\bx_0\bx_0^T))\bz=0$ implies $\bx_0^T\bz= 0$. Hence, we have $A\bz=\tilde{\lambda}\bz$. Since $\tilde{\lambda}$ has multiplicity 1, $\bz=aq$ for some $a\in\R$. In addtion,  $0= \bz^T\bx_0 = aq^T\bx_0$  implies $a=0$. That is, $\bz=0$. Hence, $A -\tilde{\lambda}I + (\gamma+\tilde{\lambda})(\bx_0\bx_0^T)$ is nonsingular and there exists a unique nonzero solution $\tilde{\bx}$ to \eqref{newton:onestep}. By multiplying \eqref{newton:onestep} by $q^T$, we have $(\gamma+\tilde{\lambda})\bx_0^T\tilde{\bx} = \gamma$, which implies that $\tilde{\bx}$ also satisfies 
$$
(A-\tilde{\lambda}I)\tilde{\bx} = 0.
$$
\qed\end{proof}

If the multiplicity of an eigenvalue $\tilde{\lambda}$ is greater than 1, then \eqref{newton:onestep} becomes singular and Proposition~\ref{prop:newton1} does not apply. However, when the multiplicity $m>1$ is known, we can construct another nonsingular system.
\begin{corollary}\label{cor:newton}
Suppose that an eigenvalue $\tilde{\lambda}$ of $A$ has multiplicity $m>1$. We choose $\bx_0,\dots, \bx_{m-1}$ uniformly at random from $S^{N-1}$ and set an $N\times m$ matrix $X_0 = [\bx_0\ \cdots \bx_{m-1}]$. With probability 1, the equation 
$$
(A-\tilde{\lambda}I + (\gamma+\tilde{\lambda})X_0X_0^T)X = \gamma X_0
$$
has a unique nonzero solution $\tilde{X}=[\tilde{\bx}_0\ \cdots \tilde{\bx}_{m-1}]$, an $N\times m$ matrix, whose columns $\tilde{\bx}_0,\dots,\tilde{\bx}_{m-1}$ constitute a basis for the eigenspace corresponding to the eigenvalue $\tilde{\lambda}$.
\end{corollary}

Moreover, Proposition~\ref{prop:newton3} below says that, regardless of an eigenvalue's multiplicity and of the symmetry of a matrix, a good estimate of the eigenvalue guarantees a good estimate of a corresponding eigenvector through the nonsingular linear system \eqref{newton:onestep}.

\begin{proposition}\label{prop:newton3}
Let $A\in{\bf M}_N(\R)$ be diagonalizable. Suppose that $\tilde{\lambda}$ is an eigenvalue of $A$. Let  $\gamma$ be such that $\gamma>0$ and $\gamma\neq -\tilde{\lambda}$.
\begin{enumerate}
\item If $\Ni(A-\tilde{\lambda}I)$, the null space of $A-\tilde{\lambda}I$, is of dimension $1$, then choosing  $\bx_0\in S^{N-1}$ uniformly at random, \eqref{newton:onestep}
has a unique nonzero solution $\tilde{\bx}$ with probability $1$, which is an eigenvector of $A$ corresponding to $\tilde{\lambda}$.
\item If $\Ni(A-\tilde{\lambda}I)$ has dimension greater that $1$, then choosing  $\bx_0\in S^{N-1}$ uniformly at random, we can see, with probability $1$,  that for $\lambda$ close enough to $\tilde{\lambda}$, 
\begin{equation}\label{newton:nonsymmetric}
(A -\lambda I +(\gamma +\lambda) \bx_0\bx_0^T)\bx = \gamma \bx_0
\end{equation}
has a unique nonzero solution $\bx_\lambda$  satisfying
$$
\eta_1|\lambda-\tilde{\lambda}|< \|\bx_\lambda-\tilde{\bx}\| < \eta_2|\lambda-\tilde{\lambda}|
$$
for some $0<\eta_1<\eta_2<\infty$, where $\tilde{\bx}\in\Ni(A-\tilde{\lambda}I)$ is uniquely determined by $\bx_0$ and $\eta_1, \eta_2$ do not depend on $\lambda$.
\end{enumerate}
\end{proposition}

\begin{proof}
The first part can be proven in the same way as we proved Proposition~\ref{prop:newton1} with a choice of $\bx_0\in S^{N-1}$ satisfying $q^T\bx_0\neq 0$ and $\tilde{q}^T\bx_0\neq 0$, where $q$ and $\tilde{q}$ are unit vectors spanning $\Ni(A-\tilde{\lambda}I)$ and $\Ni(A^T-\tilde{\lambda}I)$, respectively. 

Note that $A-\lambda I$ for $0<|\lambda -\tilde{\lambda}| <\min_{\lambda_j\neq \tilde{\lambda}}(|\lambda_j-\tilde{\lambda}|)$, is invertible and that if $\bx_0\in S^{N-1}$ is chosen uniformly at random, then $\bx_0$ is not orthogonal to $\Ni(A-\tilde{\lambda}I)$ with probability $1$. Moreover, since $\R^N = \eigspace(\tilde{\lambda})\oplus V$ with $V = \bigoplus_{\lambda_j\neq\tilde{\lambda}}\eigspace(\lambda_j)$, representing $\bx_0=q_0+r_0$ uniquely in $\R^N$, where $q_0\in\eigspace(\tilde{\lambda}) = \Ni(A-\tilde{\lambda}I)$ and $r_0\in V$, we have $q_0^T\bx_0\neq 0$ with probability $1$.  We also note that for any $0<\delta <\min_{\lambda_j\neq \tilde{\lambda}}(|\lambda_j-\tilde{\lambda}|)$, there exists $K>0$ such that 
\begin{equation}\label{nondiagonal:restriction}
\sup_{\lambda\in[\tilde{\lambda}-\delta,\tilde{\lambda}+\delta]}\|(A-\lambda I)^{-1}\|_V < K,
\end{equation}
where $\|(A-\lambda I)^{-1}\|_V$ is the operator norm of the restriction $(A-\lambda I)^{-1} :V \rightarrow V$, which is not difficult to see since $V$ is invariant under $(A-\lambda I)^{-1}$ and $(A-\lambda I)^{-1}:V\to V$ is invertible and $\|(A-\lambda I)^{-1}\|_V$ is continuous in $\lambda$ for  $\lambda\in[\tilde{\lambda}-\delta,\tilde{\lambda}+\delta]$.

So, we will fix $0<\delta<\min_{\lambda_j\neq \tilde{\lambda}}(|\lambda_j-\tilde{\lambda}|,\frac{|\gamma+\tilde{\lambda}|}{2})$ and consider $(A-\lambda I)^{-1}$ as being restricted to $V$ for $\lambda\in[\tilde{\lambda}-\delta,\tilde{\lambda}+\delta]$.

Firstly, with such an $\bx_0$, we will confirm that for $0<|\lambda-\tilde{\lambda}| < \delta$,
\begin{equation}\label{nondiag:uniq}
(A-\lambda I)\bx = (\gamma \bx_0-(\gamma+\lambda)\bx_0\bx_0^T\bx)
\end{equation}
has a unique solution $\bx_\lambda$. If there is a solution $\bx_\lambda$, then letting $$\alpha_\lambda = \gamma - (\gamma +\lambda)\bx_0^T\bx_\lambda,$$ we can see that \eqref{nondiag:uniq} becomes $(A-\lambda I)\bx_\lambda = \alpha_\lambda \bx_0$, i.e., $\bx_\lambda = \alpha_\lambda(A-\lambda I)^{-1}\bx_0$. Hence, the unique existence of $\alpha_\lambda$ confirms the unique existence of $\bx_\lambda$. Note that $\alpha_\lambda$ must satisfy
$$
\alpha_\lambda = \gamma - \alpha_\lambda(\gamma +\lambda)\bx_0^T(A-\lambda I)^{-1}\bx_0\ \Leftrightarrow\ \alpha_\lambda = \frac{\gamma}{1+(\gamma+\lambda)\bx_0^T(A-\lambda I)^{-1}\bx_0}.
$$
Together with \eqref{nondiagonal:restriction}, it is not difficult to see that for $\lambda\in[\tilde{\lambda}-\delta,\tilde{\lambda}+\delta]$,
$$
\frac{\|r_0\|}{\|A-\tilde{\lambda} I\|_V+\delta}<\|(A-\lambda I)^{-1}r_0\| < K\|r_0\|.
$$
Since $(A-\lambda I)^{-1}q_0 = \frac{1}{(\tilde{\lambda}-\lambda)}q_0$, we have
\begin{align*}
\|\bx_0^T(A-\lambda I)^{-1}\bx_0\| =&\ \Big\|\frac{1}{(\tilde{\lambda}-\lambda)}q_0^T\bx_0+\bx_0^T(A-\lambda I)^{-1}r_0\Big\|\\
\geq&\ \frac{|q_0^T\bx_0|}{|\tilde{\lambda}-\lambda|} - \|(A-\lambda I)^{-1}r_0\|>\frac{|q_0^T\bx_0|}{|\tilde{\lambda}-\lambda|} - K\|r_0\|.
\end{align*}
Noting that $|\gamma+\lambda|>\frac{|\gamma+\tilde{\lambda}|}{2}$ for $0<|\lambda-\tilde{\lambda}|<\delta$, we can see that if 
$$
0<|\tilde{\lambda}-\lambda| < \delta_0:=\min\Big(\delta, \frac{|q_0^T\bx_0|}{\beta}\Big),$$
where $\beta = \max(6K\|r_0\|, K\|r_0\|+\frac{2}{|\gamma+\tilde{\lambda}|})$, then $\|(\gamma+\lambda)\bx_0^T(A-\lambda I)^{-1}\bx_0\|>1$ and $\alpha_\lambda$ exists and  we can easily see  that the unique nonzero solution $\bx_\lambda$ to \eqref{nondiag:uniq} is represented as 
$$
\bx_\lambda = \alpha_\lambda(A-\lambda I)^{-1}\bx_0=\frac{\gamma(A-\lambda I)^{-1}\bx_0}{(1+(\gamma+\lambda)\bx_0^T(A-\lambda I)^{-1}\bx_0)}
$$
and that
\begin{align*}
&\ \Big|\frac{\alpha_\lambda}{\tilde{\lambda}-\lambda} - \frac{\gamma}{(\gamma+\tilde{\lambda})q_0^T\bx_0}\Big|\\
=&\ \Big|\frac{\gamma(q_0^T\bx_0 -(1+(\gamma+\lambda)\bx_0^T(A-\lambda I)^{-1}r_0))}{[(\lambda-\tilde{\lambda})(1+(\gamma+\lambda)\bx_0^T(A-\lambda I)^{-1}r_0)+(\gamma+\lambda)q_0^T\bx_0](\gamma+\tilde{\lambda})q_0^T\bx_0}\Big||\lambda-\tilde{\lambda}|\\
<&\ \frac{\gamma(|q_0^T\bx_0| +1+K(|\gamma+\tilde{\lambda}|+\delta_0)\|r_0\|)}{|(\gamma+\tilde{\lambda})q_0^T\bx_0|((|\gamma+\tilde{\lambda}|-\delta_0)|q_0^T\bx_0|-|\lambda-\tilde{\lambda}|(|\gamma+\tilde{\lambda}|+\delta_0)K\|r_0\|)}|\lambda-\tilde{\lambda}|.
\end{align*}
Since $|\gamma+\tilde{\lambda}|-\delta_0 > |\gamma+\tilde{\lambda}|/2$ and $|\gamma+\tilde{\lambda}|+\delta_0 < 3|\gamma+\tilde{\lambda}|/2$ and $\delta_0\leq \frac{|q_0^T\bx|}{6K\|r_0\|}$, we obtain that for $0<|\lambda-\tilde{\lambda}|<\delta_0$,
$$
(|\gamma+\tilde{\lambda}|-\delta_0)|q_0^T\bx_0|-|\lambda-\tilde{\lambda}|(|\gamma+\tilde{\lambda}|+\delta_0)K\|r_0\| > \frac{|(\gamma+\tilde{\lambda})q_0^T\bx_0|}{4}
$$
implying that
$$
\Big|\frac{\alpha_\lambda}{\tilde{\lambda}-\lambda} - \frac{\gamma}{(\gamma+\tilde{\lambda})q_0^T\bx_0}\Big|<\omega|\lambda-\tilde{\lambda}|,
$$
where 
$$
\omega = \frac{4\gamma(|q_0^T\bx_0| +1+K(3|\gamma+\tilde{\lambda}|/2)\|r_0\|)}{|(\gamma+\tilde{\lambda})q_0^T\bx_0|^2}
$$
is determined by $\bx_0$. If we set $\tilde{\bx} = \frac{\gamma}{(\gamma+\tilde{\lambda})(\bx_0^Tq_0)}q_0$, then for $0<|\lambda-\tilde{\lambda}|<\delta_0$,
\begin{align*}
\|\bx_\lambda -\tilde{\bx}\|  =&\ \Big\|\frac{\gamma (A-\lambda I)^{-1}\bx_0}{(1+(\gamma+\lambda)\bx_0^T(A-\lambda I)^{-1}\bx_0)} -\frac{\gamma}{(\gamma+\tilde{\lambda})(\bx_0^Tq_0)}q_0\Big\|\\
\leq&\ \Big\|\Big(\frac{\alpha_\lambda}{\tilde{\lambda}-\lambda} -\frac{\gamma}{(\gamma+\tilde{\lambda})(\bx_0^Tq_0)}\Big)q_0\Big\| + \|\alpha_\lambda(A-\lambda I)^{-1}r_0\|\\
<&\ \eta_2 |\lambda-\tilde{\lambda}|
\end{align*}
where $\eta_2 = \omega\|q_0\| + K\|r_0\|(\frac{\gamma}{|(\gamma+\tilde{\lambda})\bx_0^Tq_0|} + \omega \delta_0)$. 

On the the hand, we let $z_0\in V$ be such that $\bz_0 = (A-\tilde{\lambda}I)^{-1}r_0$, i.e., $(A-\tilde{\lambda})\bz_0 = r_0$ and set $\by = \bz_0 - (\frac{\bz_0^Tq_0}{\|q_0\|^2})q_0$ and $\by_0:=\frac{\by}{\|\by\|}$.  Then, it is not difficult to see that $\by_0^T\bz_0\neq 0$ and $\by_0^Tq_0 =0$ and 
$$
|\by_0^T(A-\lambda I)^{-1}r_0 - \by_0^T\bz_0| = |(\lambda-\tilde{\lambda})\by_0^T(A-\tilde{\lambda})^{-1}(A-\lambda I)^{-1}r_0|\leq K^2\|r_0\||\lambda-\tilde{\lambda}|,
$$
which implies that there exists $\delta_1>0$ such that for $0<|\lambda-\tilde{\lambda}|<\delta_1$, we have
$$
|\by_0^T(A-\lambda I)^{-1}r_0|>\frac{|\by_0^T\bz_0|}{2}\quad\text{and}\quad |\alpha_\lambda| > \frac{\gamma|\lambda-\tilde{\lambda}|}{2|(\gamma+\tilde{\lambda})\bx_0^Tq_0|}.
$$
i.e.,
$$
\|\bx_\lambda -\tilde{\bx}\| \geq |\by_0^T(\bx_\lambda-\tilde{\bx})| = |\alpha_\lambda \by^T(A-\lambda I)^{-1}r_0|> \eta_1|\lambda-\tilde{\lambda}|,
$$
with $\eta_1 =\frac{\gamma}{4}\Big|\frac{\by_0^T\bz_0}{(\gamma+\tilde{\lambda})\bx_0^Tq_0}\Big|>0$. Therefore, $0<|\lambda-\tilde{\lambda}|<\min(\delta_0,\delta_1)$ impilies
$$
\eta_1\lambda < \|\bx_\lambda -\tilde{\bx}\| < \eta_2\lambda.
$$
\qed\end{proof}

From the analysis about \eqref{newton:onestep}, we can see  a similarity between the Newton's methd \eqref{newton} and the Rayleigh quotient iteration. When generating a sequence $\{\bx_k\}$ via
$$
(A-\mu_{k} I + (\gamma+\mu_{k})\by_k\by_k^T )\bx_{k+1} = \gamma \by_k,
$$
where $\by_k = \frac{\bx_k}{\|\bx_k\|}$, we can see that \eqref{newton} is obtained with $\mu_k=\gamma(\frac{1}{\|\bx_k\|}-1)$ and the Rayleigh quotient method is obtained  with $\mu_k=\frac{\langle \bx_k, A\bx_k\rangle}{\|\bx_k\|^2}$. The same is true for solving $A\bx=\lambda B\bx$, i.e., if we use $\mu_k = \gamma(\frac{1}{\|\bx_k\|_B}-1)$, where $\|\bx_k\|_B=\sqrt{\langle \bx_k,\bx_k\rangle_B}$, we have the Newton's method \eqref{newton},
$$
\Big[\frac{1}{\gamma}A + \Big(1-\frac{1}{\|\bx_k\|_B}\Big)B + \frac{1}{\|\bx_k\|_B}\Big(\frac{B\bx_k}{\|\bx_k\|_B}\Big)\Big(\frac{B\bx_k}{\|\bx_k\|_B}\Big)^T\Big]\bx_{k+1} = \frac{B\bx_k}{\|\bx_k\|_B},
$$
whereas if we use $\mu_k = \frac{\langle \bx_k,A\bx_k\rangle}{\langle\bx_k,\bx_k\rangle_B}$, then we have the Rayleigh quotient method. We will present numerical experiments towards the end of this paper showing an interesting feature of our proposed framework, that is, our proposed framework tends to find the smallest eigenvalues due to the philosophy of our proposed framework where $F_{A,B}$ is to be minimized, whereas the Rayleigh quotient method tends to find the largest eigenvalues when both methods begin with a random initial point $\bx_0$.

\begin{remark}
We would like to mention that the same framework applies to solving eigenvalue problems involving complex Hermitian matrices and to finding singular values of $A\in{\bf M}_{M\times N}(\R)$. When it comes to singular values of $A$, we may consider either $A^TA$ or $AA^T$ in place of $A$ since $\lambda>0$ is a singular value of $A$ if and only if $\lambda^2>0$ is that of $A^TA$ or $AA^T$.
\end{remark}

\section{An eigenvalue problem on infinite dimensional spaces}
\label{sec:2}

It is interesting to see that the same framework as \eqref{eigeq:functional} applies to eigenvalue problems on infinite dimensional spaces such as the Sturm-Liouville eigenvalue problem, the eigenvalue problem of self-adjoint elliptic operators, etc. We will present one such application of finding an eigenfunction of a self-adjoint uniformly elliptic operator corresponding to the smallest eigenvalue, which is an infinite dimensional version of a real symmetric and positive definite matrix.

%
%

\subsection{Symmetric uniformly elliptic operators}
\label{subsec:9}


Let $\Omega$ be a bounded open subset of $\R^d$ with Lipschitz boundary $\partial\Omega$. We will denote by $L$ a symmetric uniformly elliptic operator defined by
\begin{equation}\label{eigeq:operator}
Lu = -\sum_{i,j=1}^d\partial_i(a_{i,j}\partial_j u) +cu,
\end{equation}
where $a_{i,j},c\in L^\infty(\Omega)$, $i,j=1,\dots,d$, are such that $a_{i,j}(x) = a_{j,i}(x)$ a.e., and $c(x)\geq 0$ a.e., and there exists $0<\alpha\leq \beta<\infty $ such that  for a.e. $x\in\Omega$ and for $\xi\in\R^d$, $$\alpha|\xi|^2\leq \sum_{i,j=1}^da_{i,j}(x)\xi_i\xi_j \leq \beta|\xi|^2.$$ The problem that we are interested in is to find $\varphi\in H_0^1(\Omega)$ that solves
\begin{equation}\label{eigpr:1}
\begin{cases}
L\varphi = \lambda \varphi,&\text{ in }\ \Omega,\\
\ \ \varphi = 0,&\text{ on } \partial\Omega.
\end{cases}
\end{equation}
It is known that the eigenvalues $\lambda_1,\lambda_2,\dots$ of $L$ are nonnegative and we may have them in an increasing order, that is, $$0< \lambda_1<\lambda_2\leq\cdots\leq \lambda_3\leq \cdots.$$
Therefore, a natural question to ask is to find the smallest eigenvalue $\lambda_1$ of $L$ and its corresponding eigenfuction. 

\begin{definition}
$\varphi\in H_0^1(\Omega)$ is a weak solution of \eqref{eigpr:1} if for any $\psi\in C_0^\infty(\Omega)$,
$$
\int_\Omega \sum_{i,j=1}^da_{i,j}(x)\partial_j \varphi(x)\partial_i \psi(x)dx + \int_\Omega (c(x)-\lambda)\varphi(x)\psi(x)dx =0.
$$
\end{definition}

Then, it is natural from the discussion in the previous sections that we want to define a functional $F_L$ with $\gamma>0$ by
\begin{align}\label{eigeq:functionalL}
F_L(u) =&\ \frac{1}{2}\int_\Omega \sum_{i,j=1}^da_{i,j}(x)\partial_ju(x)\partial_iu(x) dx + \frac{1}{2}\int_\Omega c(x)|u(x)|^2dx\\
&\  + \frac{\gamma}{2}\int_\Omega |u(x)|^2dx - \gamma\Big(\int_\Omega |u(x)|^2dx\Big)^{\frac{1}{2}}\nonumber
\end{align}
and solve the following minimization problem
\begin{equation}\label{eigeq:minimizeL}
\min_{u\in H_0^1(\Omega)}F_L(u)
\end{equation}
and investigate the relationship between \eqref{eigpr:1} and \eqref{eigeq:minimizeL}. The existence of a minimizer of the problem \eqref{eigeq:minimizeL} is obvious by the standard method using the compact embedding theorem by Rellich-Kondrachov. Moreover, we can observe the same characteristics of \eqref{eigeq:functionalL} as those of \eqref{eigeq:functional}. For theorems and lemmas that follow, we will omit their proofs because they are essentially the same as what we presented in the finite dimensional case.

\begin{lemma}\label{eiglem:3}
The set of nonzero critical points of $F_L$ is 
$$
\Big\{\varphi \in H_0^1(\Omega) : L\varphi = \lambda\varphi\text{ in }\Omega\ \text{ and }\ \|\varphi\|_{L^2(\Omega)} = \frac{\gamma}{\gamma+\lambda}\Big\}.
$$
and 
$$
\min_{u\in H_0^1(\Omega)}F_L(u) = -\frac{\gamma^2}{2(\gamma + \lambda_*)},
$$
where $\lambda_*>0$ is the smallest eigenvalue of $L$.
\end{lemma}

\begin{theorem}
Any local minimizer of \eqref{eigeq:functionalL} is a global minimizer.
\end{theorem}
\subsubsection{Corresponding parabolic PDEs}
\label{subsec:10}

Corresponding to the uniformly elliptic PDEs of the form \eqref{eigpr:1}, we will consider the following parabolic PDE:  for $T\in(0,\infty)$, we solve
\begin{equation}\label{model:pde}
\begin{cases}
\frac{\partial u}{\partial t} &=\ -L u -\gamma\Big(1-\frac{1}{\|u(t)\|_2}\Big)u\ \text{ in }\ \Omega_T:= \Omega\times(0,T],\\
u &=\ 0\ \text{ on }\ \partial\Omega_T :=\partial\Omega\times [0,T],\\
u(0) &=\ u_0 \neq 0\text{ in } H_0^1(\Omega),
\end{cases}
\end{equation}
where $\|u(t)\|_2 = (\int_\Omega |u(x,t)|^2dx)^\frac{1}{2}$. Due to the condition $c\geq 0$ in $\Omega$, $L$ is positive definite. In the context, we will use $\langle\cdot,\cdot,\rangle$ for both inner products in $L^2(\Omega)$ and in $\R^N$ for $N\in\N$.
Note that the partial differential equation in \eqref{model:pde} is the formal gradient flow of a functional $F_L$ in \eqref{eigeq:functionalL}, i.e., 
$$
\frac{\partial u}{\partial t} = -\nabla F_L(u).
$$
\begin{theorem}\label{thm:parabolicpde}
The problem \eqref{model:pde} has a unique weak solution $u$ for any $T\in(0,\infty)$.
 Moreover, if  $\psi_1$ is an eigenfunction of $L$ corresponding to the smallest eigenvalue $\lambda_1>0$ with $\|\psi_1\|_2 = 1$, and if $\langle u_0,\psi_1\rangle\neq 0$, then the solution $u$ satisfies
$$
u(t) \rightarrow \frac{\gamma}{\gamma+\lambda_1}\vec{v} \text{ in } L^2(\Omega) \text{ as } t \rightarrow\infty\ \text{ for some } \vec{v}\in \{\pm\psi_1\}.
$$
\end{theorem}

Before proving Theorem~\ref{thm:parabolicpde}, we will present an ODE version of \eqref{model:pde}, which will be used when proving Theorem~\ref{thm:parabolicpde}. 

For $A\in\symm_{N,p}(\R)$, as we did previously, if we consider a diagonalization of $A$, $A=Q\Lambda_N Q^T$, where  $Q$ is an orthogonal matrix and  $\Lambda_N = diag(\lambda_1,\dots,\lambda_N)$, $0< \lambda_1=\cdots=\lambda_p<\lambda_{p+1}\leq \cdots\leq\cdots\leq \lambda_N$ for some $1\leq p<N$,  we have
\begin{equation}\label{equiv:diagonal}
F_A(\vec{x}) = F_{\Lambda_N}(\vec{y}),\ \text{ with } \vec{y} = Q^T\vec{x}.
\end{equation}
The gradient descent flow associated with the functional $F_{\Lambda_N}$ is
$$
\frac{d}{dt}\Phi_N(t) = -\nabla F_{\Lambda_N}(\Phi_N(t)),
$$
and we are interested in the existence of a solution of \eqref{equiv:odes} below, which is to solve  on $(0,\infty)$
\begin{equation}\label{equiv:odes}
\begin{cases}
\frac{d}{dt}\Phi_N(t)& =\ \ -\nabla F_{\Lambda_N}(\Phi_N(t)) = -(\Lambda_N+\gamma I)\Phi_N(t) +\frac{\gamma}{\|\Phi_N(t)\|}\Phi_N(t),\\ 
\Phi_N(0) &  \in \ \  \R^N,\ \phi_j(0)\neq 0\ \text{ for some } 1\leq j\leq p,
\end{cases}
\end{equation}
where $\Phi_N = \Phi_N(t) = [\phi_1(t)\ \cdots\ \phi_N(t)]^T$ and $ \|\Phi_N(t)\| = \sqrt{\sum_{k=1}^N \phi_k^2(t)}$.

\begin{theorem}\label{thm:existenceODE}
There exist a unique solution $\Phi_N\in C^1([0,\infty))$ of \eqref{equiv:odes} and  $\vec{v}\in\R^N$ such that
$$\Phi_N(t)\rightarrow \vec{v}\ \text{ as } \ t\rightarrow\infty,$$ where $\vec{v}$ is an eigenvector of $\Lambda_N$ corresponding to the eigenvalue $\lambda_1$ with $\|\vec{v}\| = \frac{\gamma}{\gamma+\lambda_1}$ and $\langle \vec{v},\vec{e}_j\rangle\neq 0$. Note that $\{\vec{e}_1,\dots,\vec{e}_N\}$ is the standard basis for $\R^N$.

\end{theorem}
\begin{proof}
The existence and uniqueness of a solution $\Phi_N$ on $[0,\infty)$ is easily guaranteed by the theory of ODEs once we establish a lower bound $\omega>0$ for $\|\Phi_N(t)\|$, $t\in[0,\infty)$. Due to  $\phi_j(0)\neq 0$, i.e., $\|\Phi_N(0)\|> 0$, a solution exists and is unique on $[0,\epsilon]$ with $\epsilon<<1$. Then, by setting $$T_{max}= \sup\{T\in(0,\infty) : \text{a solution } \Phi_N \text{ exists and is unique on } [0,T] \},$$ we know that $T_{max}\geq \epsilon$. Note that for $t\in (0,T_{max})$ and for $1\leq k\leq N$, we have
\begin{align*}
\frac{1}{2}\frac{d}{dt}\phi_k^2(t) =&\ -(\lambda_k+\gamma)\phi_k^2(t) + \frac{\gamma}{\|\Phi_N(t)\|}\phi_k^2(t)\\
>&\ -(\lambda_N+\gamma)\phi_k^2(t),
\end{align*}
resulting in $\phi_k(t)\neq 0$ for $t\geq0$ if and only if $\phi_k(0)\neq 0$. 

Let $\omega = \frac{1}{2}\min(\frac{\gamma}{\gamma+\lambda_N}, \|\Phi_N(0)\|)$. Suppose $$\{t\in(0,T_{max}) : \|\Phi_N(t)\|\leq\omega\}\neq \emptyset.$$ Then, $T_\omega = \inf\{t\in(0,T_{max}) : \|\Phi_N(t)\|\leq\omega\}$ exists.
Since there exists $\delta>0$ such that $\omega<\|\Phi_N(t)\|< \frac{3}{2}\omega$ on $(T_\omega-\delta, T_\omega)$, we have
$$
\frac{1}{2}\frac{d}{dt}\|\Phi_N(t)\|^2  \geq \frac{1}{3}(\lambda_N + \gamma)\|\Phi_N(t)\|^2>0\ \text{ on } (T_\omega-\delta,T_\omega),
$$
that is, $\|\Phi_N(t)\|$ is increasing in $(T_\omega-\delta,T_\omega)$ and $\lim_{t\rightarrow T_\omega} \|\Phi_N(t)\| > \omega = \|\Phi_N(T_\omega)\|.$ This is a contradiction. Therefore, we have
$$
\inf_{t\in(0,T_{max})}\|\Phi_N(t)\| >\omega,
$$
which eventually implies that $T_{max} = \infty$.

In addition, we note that
$$
\frac{d}{dt}F_{\Lambda_N}(\Phi_N) =\Big\langle\frac{d}{dt}\Phi_N, \Lambda_N\Phi_N + \gamma \Phi_N -\frac{\gamma}{\|\Phi_N\|}\Phi_N\Big\rangle = -\Big\|\frac{d\Phi_N}{dt}\Big\|^2
$$
and that since $F_{\Lambda_N}(\Phi_N)$ is bounded below, there exists $y\in\R$ such that
\begin{equation}\label{simplemodel:min}
\lim_{t\rightarrow\infty} F_{\Lambda_N}(\Phi_N(t)) = y
\end{equation}
implying that 
$$
\lim_{t\rightarrow\infty}\Big\|\frac{d\Phi_N}{dt}\Big\|^2  = 0\ \Leftrightarrow\ \lim_{t\rightarrow\infty}\Big\|\Lambda_N\Phi_N + \gamma\Big(1-\frac{1}{\|\Phi_N\|}\Big)\Phi_N\Big\|^2 = 0.
$$
That is,  for each $k=1,2,\dots,N$,
\begin{equation}\label{limit1}
\lim_{t\rightarrow\infty} \Big(\lambda_k+\gamma -\frac{\gamma}{\|\Phi_N(t)\|}\Big)^2\phi_k^2(t) =0.
\end{equation}
Since 
$$
0 = \lim_{t\rightarrow\infty}\Big\|\Lambda_N\Phi_N + \gamma\Big(1-\frac{1}{\|\Phi_N\|}\Big)\Phi_N\Big\|\geq \lim_{t\rightarrow\infty}\Big|\|(\Lambda_N+\gamma I)\Phi_N\| - \gamma\Big|,
$$
we have 
\begin{equation}\label{limit2}
\lim_{t\rightarrow\infty}\sum_{k=1}^N(\lambda_k+\gamma)^2\phi_k^2(t) = \gamma^2,
\end{equation}
which implies that there exist $k_1$ and $\{t_n\}$ such that $t_n\rightarrow\infty$ as $n\rightarrow\infty$ and $$\lim_{n\rightarrow\infty}\phi_{k_1}^2(t_n) >0.$$ Unless $\lim_{n\rightarrow\infty}\phi_{k_1}^2(t_n) = \frac{\gamma^2}{(\gamma +\lambda_{k_1})^2}$, there exists $k_2\neq k_1$ such that $$\lim_{n\rightarrow\infty}\phi_{k_2}^2(t_n)>0$$ by taking a subsequence of $\{t_n\}$ if necessary. Then, \eqref{limit1} implies 
$$
\lambda_{k_1}= \lambda_{k_2}\ \text{ and }\ \lim_{n\rightarrow\infty} \|\Phi_N(t_n)\| = \frac{\gamma}{\gamma+\lambda_{k_1}}.
$$
We may repeat this process finitely many times to have
$$
\lim_{n\rightarrow\infty}\sum_{\{k : \lambda_k = \lambda_{k_1}\}}\phi_k^2(t_n) = \frac{\gamma^2}{(\gamma + \lambda_{k_1})^2},\quad \lim_{n\rightarrow\infty}\sum_{\{k : \lambda_k\neq \lambda_{k_1}\}}\phi_k^2(t_n) = 0.
$$
Suppose that there exist $l$ and $\{s_n\}$ such that $\lambda_l \neq \lambda_{k_1}$ and $s_n\rightarrow\infty$ as $n\rightarrow\infty$ and 
$$
\lim_{n\rightarrow\infty}\phi_{l}^2(s_n) > 0.
$$
Then, as was done above, with a subsequence of $\{s_n\}$ if necessary, we have $\lim_{n\rightarrow\infty}\|\Phi_N(s_n)\| =\frac{\gamma}{\gamma+\lambda_l}$ and
$$
\lim_{n\rightarrow\infty}\sum_{\{k : \lambda_k = \lambda_l\}}\phi_k^2(s_n) = \frac{\gamma^2}{(\gamma + \lambda_{l})^2},\quad \lim_{n\rightarrow\infty}\sum_{\{k : \lambda_k\neq \lambda_l\}}\phi_k^2(s_n) = 0
$$
and \eqref{simplemodel:min} implies that
$$
\lim_{n\rightarrow\infty}F_{\Lambda_N}(\Phi_N(t_n)) = -\frac{\gamma^2}{2(\gamma+\lambda_{k_1})} = y = \lim_{n\rightarrow\infty}F_{\Lambda_N}(\Phi_N(s_n)) =  -\frac{\gamma^2}{2(\gamma+\lambda_{l})},
$$
which is a contradiction. Therefore, we conclude that $\lim_{t\rightarrow\infty}\|\Phi_N(t)\| =\frac{\gamma}{\gamma+\lambda_{k_1}}$ and 
$$
\lim_{t\rightarrow\infty}\sum_{\{k : \lambda_k = \lambda_{k_1}\}}\phi_k^2(t) = \frac{\gamma^2}{(\gamma + \lambda_{k_1})^2},\quad \lim_{t\rightarrow\infty}\sum_{\{k : \lambda_k\neq \lambda_{k_1}\}}\phi_k^2(t) = 0.
$$

We will now claim that $\lambda_{k_1} = \lambda_1$. Suppose that $\lambda_{k_1}>\lambda_1$, i.e., $k_1>p$. Then, since $\frac{\gamma(2\gamma+\lambda_1+\lambda_{k_1})}{2(\gamma + \lambda_1)(\gamma+\lambda_{k_1})}>\frac{\gamma}{(\gamma+\lambda_{k_1})}$, there exists $T>0$ such that for $t>T$, $$\|\Phi_N(t)\|<\frac{\gamma(2\gamma+\lambda_1+\lambda_{k_1})}{2(\gamma + \lambda_1)(\gamma+\lambda_{k_1})},$$ and we have
$$
\frac{1}{2}\frac{d}{dt}\sum_{k=1}^p\phi_k^2(t) =\Big( -(\lambda_1+\gamma) + \frac{\gamma}{\|\Phi_N(t)\|}\Big)\sum_{k=1}^p\phi_k^2(t) \geq\Big( \frac{\lambda_{k_1}-\lambda_1}{2\gamma+\lambda_1+\lambda_{k_1}}\Big)\sum_{k=1}^p\phi_k^2(t)
$$
for $t>T$, which results in $\|\Phi_N(t)\|\rightarrow \infty$ as $t\rightarrow\infty$. This is a contradiction. Therefore, we have $\lambda_{k_1} = \lambda_1$ and $\lim_{t\rightarrow\infty}\sum_{k=p+1}^N\phi_k^2(t) = 0$ and
\begin{equation}\label{equiv:odes_conv}
\lim_{t\rightarrow\infty}\|\Phi_N(t)\|^2 =\lim_{t\rightarrow\infty}\sum_{k=1}^p\phi_k^2(t) = \frac{\gamma^2}{(\gamma + \lambda_{1})^2}.
\end{equation}

Lastly, we will claim that there exists $\vec{v}\in\R^N$ such that $\|\vec{v}\| = \frac{\gamma}{\gamma+\lambda_1}$, and $\langle\vec{v}, \vec{e}_j\rangle \neq 0$, and 
$$
\Phi_N(t)\rightarrow \vec{v}\ \text{ as } \ t\rightarrow\infty.
$$
Note that for $k=1,\dots, p$, if $\phi_k(0)\neq 0$, then
$$
\frac{1}{2\phi_k^2(t)}\frac{d}{dt}\phi_k^2(t)= -(\lambda_1+\gamma) +\frac{\gamma}{\|\Phi_N(t)\|}\ \text{ for }\ t\geq 0,
$$
implying that for $T\in(0,\infty)$,
$$
\ln(\phi_k^2(T)) - \ln(\phi_k^2(0)) = 2\int_0^T \Big(-(\lambda_1+\gamma)+\frac{\gamma}{\|\Phi_N(t)\|}\Big)dt =: 2\Psi(T).
$$
From \eqref{equiv:odes_conv}, we can see that
$$
\sum_{k=1}^p\phi_k^2(T) = e^{2\Psi(T)}\sum_{k=1}^p\phi_k^2(0) \rightarrow \Big(\frac{\gamma}{\gamma+\lambda_1}\Big)^2\ \text{ as }\ T\rightarrow\infty,
$$
which implies
$$
\Psi_\infty = \lim_{t\rightarrow\infty}\Psi(t)= \ln\Big(\frac{\gamma}{\gamma+\lambda_1}\Big)-\frac{1}{2}\ln\Big(\sum_{k=1}^p\phi_k^2(0)\Big).
$$
Hence, for $1\leq k\leq p$, 
$$
\phi_k^2(t)\rightarrow v_k^2\ \text{ as }\ t\rightarrow\infty,
$$
where $v_k = \phi_k(0)e^{\Psi_\infty}$. Noting that $\phi_k(t)\phi_k(0) >0$ for all $t\geq 0$ unless $\phi_k(0) = 0$, we can easily see that
$$
\phi_k(t)\rightarrow v_k\ \text{ as }\ t\rightarrow\infty.
$$
If we set
$$
\vec{v} = [v_1\ \cdots\ v_p\ 0\ \cdots\ 0]^T,
$$
then $\vec{v}$ is an eigenvector of $\Lambda_N$ corresponding to $\lambda_1$ and $\langle \vec{v},\vec{e}_j\rangle =v_j\neq 0$ since $\phi_j(0)\neq 0$ for some $1\leq j\leq p$, and 
$$
\Phi_N(t) \rightarrow \vec{v}\ \text{ as }\ t\rightarrow\infty
$$
and $\|\vec{v}\| = \frac{\gamma}{\gamma+\lambda_1}$.
\qed\end{proof}
We will now give a proof of Theorem~\ref{thm:parabolicpde}.

\begin{proof}{(\bf Theorem~\ref{thm:parabolicpde})}
The proof is inspired by the Galerkin's method. Note that we can find an orthonormal basis $\{\psi_k\}_{k\in\N}$ for $L^2(\Omega)$ such that $\psi_k$ is an eigenfunction of $L$ corresponding to the $k^{th}$ smallest eigenvalue $\lambda_k$ with $\psi_k\in H_0^1(\Omega)$. Then, we let $V_N$, $N\in\N$, be the subspace of $L^2(\Omega)$ spanned by $\{\psi_1,\dots, \psi_N\}$ and let $P_N$ be the projection of $L^2(\Omega)$ onto $V_N$.

Suppose that $u_0\in H_0^1(\Omega)$ is given and satisfies $\langle u_0,\psi_1\rangle \neq 0$. We  consider
\begin{equation}\label{simplemodel}
\begin{cases}
\frac{\partial u}{\partial t} &=\ -L u -\gamma\Big(1-\frac{1}{\|u(t)\|_2}\Big)u\ \text{ in }\ \Omega_T,\\
u &=\ 0\ \text{ on }\ \partial\Omega_T,\\
u(0) &=\ P_N(u_0).
\end{cases}
\end{equation}
If $u_N\in L^2([0,T];H_0^1(\Omega))$ with $\frac{d}{dt}u_N\in L^2([0,T];H^{-1}(\Omega))$ is a solution of \eqref{simplemodel}, then for $k=1,\dots, N$, we have
$$
\frac{d}{dt}\langle u_N, \psi_k\rangle = -(\lambda_k+\gamma)\langle u_N,\psi_k\rangle + \frac{\gamma}{\|u_N(t)\|_2} \langle u_N,\psi_k\rangle
$$
and
$$
\frac{1}{2}\frac{d}{dt}\langle u_N, \psi_k\rangle^2 = -(\lambda_k+\gamma)\langle u_N,\psi_k\rangle^2 + \frac{\gamma}{\|u_N(t)\|_2} \langle u_N,\psi_k\rangle^2
$$
with $\langle u_N(0),\psi_k\rangle = \langle P_N(u_0),\psi_k\rangle$. Considering $\phi_k(t) = \langle u_N(t),\psi_k\rangle$, $k=1,\dots,N$, we can see that $\phi_1,\dots,\phi_N$ solve \eqref{equiv:odes}.  Since $\Phi_N =[\phi_1\ \cdots\ \phi_N]^T$ satisfies that for $t\in[0,\infty)$,
\begin{equation}\label{solutionbound}
\|u_N(t)\|_2\geq \|\Phi_N(t)\|> \omega = \frac{1}{2}\min\Big(\frac{\gamma}{\gamma+\lambda_N}, \|\Phi_N(0)\|\Big).
\end{equation}
Therefore, the existence of a solution to \eqref{simplemodel} can be obtained by a linear system of ODEs  \eqref{equiv:odes} given in the Appendix below, with the initial condition $$\Phi_N(0) = [\langle P_N(u_0),\psi_1\rangle\ \cdots\ \langle P_N(u_0),\psi_N\rangle]^T,$$ and setting
$$
u_N(x,t) =  \phi_1(t)\psi_1(x) + \cdots +\phi_N(t)\psi_N(x).
$$
Then, $u_N$ is a weak solution of \eqref{simplemodel} such that
$$
u_N(t)\rightarrow \frac{\gamma}{\gamma+\lambda_1}\vec{v}\ \text{ in } L^2(\Omega)\ \text{ as }\ t\rightarrow\infty,
$$
where $\phi_1,\dots,\phi_N$ satisfy \eqref{equiv:odes} and $\vec{v}$ is either $\psi_1$ or $-\psi_1$. 


Suppose now that $u_{N,1}, u_{N,2}$ are two such solutions of \eqref{simplemodel}. Let $v_N = u_{N,1} - u_{N,2}$. Then, we have
\begin{equation}\label{simplemodel:v}
\frac{\partial v_N}{\partial t} = -Lv_N - \gamma (f(u_{N,1}) - f(u_{N,2})),
\end{equation}
where $f(u(x,t)) = u(x,t) - \frac{u(x,t)}{\|u(t)\|_2}$. Since $f$ on $L^2(\Omega)$ is Lipschitz with a Lipschitz constant $\mu= (\frac{2}{\omega}-1)$ on $\{u : \|u\|_2\geq \omega\}$ and the two solutions $u_{N,1}, u_{N,2}$ satisfy \eqref{solutionbound} for all $t\geq 0$, by taking the inner product on $L^2(\Omega)$ with $v_N$ on both sides of \eqref{simplemodel:v}, we have
$$
\frac{1}{2}\frac{d}{dt}\langle v_N(t),v_N(t)\rangle  \leq  \langle -Lv_N(t), v_N(t)\rangle+\gamma \mu\langle v_N(t),v_N(t)\rangle,
$$
which implies 
\begin{equation}\label{intermediatesolution}
\frac{1}{2}\frac{d}{dt}\langle v_N,v_N\rangle  -\gamma \mu\langle v_N,v_N\rangle \leq  -\langle Lv_N, v_N\rangle \leq 0.
\end{equation}
Therefore, $\|v_N(0)\|_2= 0$ implies $u_{N,1} = u_{N,2}$ for a.e. $(x,t)\in \Omega\times[0,T]$ for any $T\in(0,\infty)$. In fact, we know that
$$
u_N\in C^1([0,\infty);H_0^1(\Omega))
$$
as well as
$$
u_N\in L^2([0,\infty);H_0^1(\Omega)),\quad \frac{d}{dt}u_N\in L^2([0,\infty);H^{-1}(\Omega)).
$$

We now solve \eqref{model:pde}. Firstly, we fix $k_0\in\N$ so that $\frac{\|u_0\|_2}{\sqrt{2}}<\|P_{k_0}(u_0)\|_2$ with $\lambda_{k_0} < \lambda_{k_0+1}$ and obtain the solution $u_N$ of \eqref{simplemodel} with $N>k_0$. We may write $u_N$ as
$$
u_N(t) = \phi_{1,N}(t)\psi_1 + \cdots + \phi_{N,N}(t)\psi_N.
$$

Let $\varphi_{1,N}(t) = \phi_{1,N}^2(t) + \cdots + \phi_{k_0,N}^2(t)$ and $\varphi_{2,N}(t) = \phi_{k_0+1,N}^2(t) + \cdots + \phi_{N,N}^2(t)$. Then, we have
\begin{align}
\frac{1}{2}\frac{d\varphi_{1,N}}{dt} &\ \geq -(\lambda_{k_0}+\gamma)\varphi_{1,N} + \frac{\gamma}{\|\Phi_N(t)\|}\varphi_{1,N},\label{exactmodel:ode1}\\
\frac{1}{2}\frac{d\varphi_{2,N}}{dt} &\ \leq -(\lambda_{k_0+1} +\gamma)\varphi_{2,N} + \frac{\gamma}{\|\Phi_N(t)\|}\varphi_{2,N}\label{exactmodel:ode2}.
\end{align}
This implies that 
\begin{align*}
\varphi_{1,N}(t)  \geq&\ \varphi_{1,N}(0)e^{-2(\lambda_{k_0}+\gamma)t +\int_0^t\frac{2\gamma}{\|\Phi_N(s)\|}ds} \\
\geq& \varphi_{2,N}(0)e^{-2(\lambda_{k_0+1}+\gamma)t +\int_0^t\frac{2\gamma}{\|\Phi_N(s)\|}ds}\geq \varphi_{2,N}(t)
\end{align*}
i.e.,
\begin{equation}\label{exactmodel:ode3}
\frac{\varphi_{1,N}(t)}{\varphi_{2,N}(t)}\geq \frac{\varphi_{1,N}(0)}{\varphi_{2,N}(0)}e^{2(\lambda_{k_0+1}-\lambda_{k_0})t} > e^{2(\lambda_{k_0+1}-\lambda_{k_0})t}
\end{equation}
due to $\varphi_{1,N}(0) >\frac{\|u_0\|_2^2}{2}> \varphi_{2,N}(0)$.  Let 
$$
\begin{cases}
M_1 &=\  2\max\Big(\frac{\gamma}{\gamma+\lambda_1},\|u_0\|_2\Big),\\
M_2 &=\ \frac{1}{\sqrt{2}}\min\Big(\frac{\gamma}{\lambda_{k_0}+\gamma}, \|u_0\|_2\Big).
\end{cases}
$$
Note that 
$$
\frac{1}{2}\frac{d\varphi_{1,N}}{dt} \leq -(\lambda_1+\gamma)\varphi_{1,N} + \frac{\gamma}{\|\Phi_N(t)\|}\varphi_{1,N}
$$
implies  $\varphi_{1,N}(t) \leq \|\Phi_N(t)\|^2 < M_1^2$. And \eqref{exactmodel:ode3} implies
\begin{equation}\label{exactmodel:ineq2}
\varphi_{2,N}(t)< M_1^2e^{-2(\lambda_{k_0+1}-\lambda_{k_0})t}.
\end{equation}

Since we already saw that $\phi_{j,N}(t)$, $j=1,2,\dots,N$, exist for all $t\in(0,\infty)$, if we suppose $$t_0=\inf\{t\in(0,\infty) | \varphi_{1,N}(t) =M_2^2\}<\infty,$$ then from \eqref{exactmodel:ode1} and \eqref{exactmodel:ode3},  we have
$$
\|\Phi_N(t_0)\|^2 = \varphi_{1,N}(t_0) + \varphi_{2,N}(t_0) < M_2^2(1+e^{-2(\lambda_{k_0+1}-\lambda_{k_0})t_0})<2M_2^2\leq \Big(\frac{\gamma}{\lambda_{k_0}+\gamma}\Big)^2.
$$
Therefore, there exists $\delta>0$ such that for $t\in(t_0-\delta,t_0)$,
$$
\|\Phi_N(t)\| < \frac{\gamma}{\lambda_{k_0}+\gamma},
$$
which implies that
$$
\frac{1}{2}\frac{d\varphi_{1,N}}{dt}(t)  \geq -(\lambda_{k_0}+\gamma)\varphi_{1,N}(t) + \frac{\gamma}{\|\Phi_N\|}\varphi_{1,N}(t) >0\ \text{ on } (t_0-\delta,t_0).
$$
This is a contradiction since $\varphi_{1,N}(t) > M_2^2$ for $t\in[0,t_0)$. Therefore, we end up with
\begin{equation}\label{exactmodel:ineq1}
M_2^2 < \varphi_{1,N}(t) < M_1^2\ \text{ for } t\geq 0.
\end{equation}
Note that \eqref{exactmodel:ineq1} implies
$$
\inf_{t\geq 0}\|\Phi_N(t)\|\geq M_2
$$
That is, the solution $u_N$ for $N>k_0$ satisfies
\begin{equation}\label{exactmodel:estimate1}
\inf_{t\geq 0} \|u_N(t)\|_2\geq M_2.
\end{equation}

We fix $T\in(0,\infty)$ and consider a sequence of solutions $\{u_N(t)\}_{N > k_0}$, where $u_N$ is the solution of  \eqref{simplemodel} with $u_N(0) = P_N(u_0)$. Using $v = u_N - u_M$ with $N,M>k_0$ in place of $v_N$ in \eqref{intermediatesolution}, and noting that we may choose $\omega = M_2$ for $\mu = (\frac{2}{\omega}-1)$, we obtain
\begin{equation}\label{exactmodel:estimate2}
\|u_N(t) - u_M(t)\|_2 \leq e^{\gamma \mu t}\|u_N(0) - u_M(0)\|_2\leq  e^{\gamma \mu T}\|u_N(0) - u_M(0)\|_2\ \text{ for } t\in[0,T],
\end{equation}
which implies that for any $m\in\N$, $\{\phi_{m,N}(t)\}_{N>k_0}$ converges uniformly to $\phi_{m,*}(t)$ in $[0,T]$ and
$$
\|\Phi_N(t)\| \rightarrow \|\Phi_*(t)\| \text{ uniformly on } [0,T] \text{ as }\ N\rightarrow\infty,
$$
where $\phi_{m,N}(t) = \langle u_N(t), \psi_m\rangle$ and $\|\Phi_*(t)\| = \sqrt{\sum_{m=1}^\infty \phi_{m,*}^2(t)}$.  Hence,  for all $m$, we have
$$
\frac{d}{dt}\phi_{m,*} = -(\lambda_m+\gamma)\phi_{m,*} + \frac{\gamma}{\|\Phi_*\|}\phi_{m,*}.
$$
In addition, for $k_0< m\leq N$, if we consider $$\omega_{m,N}(t) = \phi_{m,N}^2(t) + \cdots + \phi_{N,N}^2(t),$$ then we can obtain, by the same argument for \eqref{exactmodel:ineq2},
\begin{equation}\label{exactmodel:ineq3}
\omega_{m,N}(t) < M_1^2e^{-2(\lambda_{m}-\lambda_{k_0})t},\ t\geq 0
\end{equation}
implying $|\phi_{m,N}(t)|< M_1e^{-(\lambda_{m}-\lambda_{k_0})t}$ on $[0,T]$ and, eventually, on $(0,\infty)$. By a slight modification on \eqref{exactmodel:ineq2}, we can see that even for each $2\leq m\leq k_0$, there exists $\zeta_m>0$ such that $|\phi_{m,N}(t)|< \zeta_mM_1e^{-(\lambda_{m}-\lambda_{1})t}$ on $[0, \infty)$, i.e., with $\eta_{k_0} =\max_{k=2,\dots,k_0}(\zeta_k)$,
\begin{equation}\label{mNbound}
|\phi_{m,N}(t)| <\begin{cases}\eta_{k_0}M_1e^{-(\lambda_{m}-\lambda_{1})t}, &\ 2\leq m\leq k_0,\\ M_1e^{-(\lambda_{m}-\lambda_{k_0})t}, &\ m>k_0.\end{cases}
\end{equation}
This implies that
$$
\Big\{\sum_{m=1}^N \lambda_m\phi_{m,N}^2\Big\}_{N\in\N}\text{ converges uniformly to }\sum_{m=1}^\infty\lambda_m\phi_{m,*}^2\ \text{ in } [0,T] \text{ as } N\rightarrow\infty.
$$
Hence, by defining
$$
u_*(x,t) = \sum_{m=1}^\infty \phi_{m,*}(t)\psi_m(x),
$$
we can easily see that $u_*\in L^2([0,T];H_0^1(\Omega))$ and $$\frac{d}{dt}u_* =\sum_{m=1}^\infty \frac{d\phi_{m,*}}{dt}(t)\psi_m(x) \in L^2([0,T];H^{-1}(\Omega))$$ and that for a.e. $t\in[0,T]$, and for each $v\in H_0^1(\Omega)$,
\begin{align*}
&\ \int_\Omega \frac{\partial}{\partial t}u_*(x,t)v(x)dx +\int_\Omega \sum_{i,j=1}^da_{i,j}(x)\partial_ju_*(x)\partial_iv(x)dx +\int_\Omega c(x)u_*(x)v(x)dx\\
&\ \gamma\int_\Omega u_*(x)v(x)dx -\frac{\gamma}{\|u_*(t)\|_2}\int_\Omega u_*(x)v(x)dx =0
\end{align*}
with $u_*(0) = u_0\  \text{ in } L^2(\Omega).$ Therefore, $u_*$ is a weak solution of \eqref{model:pde} for any $T\in(0,\infty)$. Next, using the functional in \eqref{eigeq:functionalL}, we define $f(t)$ for $t\in(0,\infty)$ by
\begin{align*}
f(t) = F_L(u_*) =&\ \frac{1}{2}\int_\Omega \sum_{i,j=1}^da_{ij}(x)\partial_ju_*(x,t)\partial_iu_*(x,t)dx  +\frac{1}{2}\int_\Omega c(x)|u_*(x,t)|^2dx\\
&\ +\frac{\gamma}{2}\int_\Omega |u_*(x,t)|^2dx - \gamma\Big(\int_\Omega|u_*(x,t)|^2dx\Big)^\frac{1}{2},
\end{align*}
Note that
$$
f(t) = \frac{1}{2}\sum_{m=1}^\infty \lambda_m\phi_{m,*}^2(t) +\frac{\gamma}{2}\sum_{m=1}^\infty\phi_{m,*}^2(t) - \gamma\|\Phi_*(t)\|.
$$
Since $\phi_{m,*}^2(t)$ decays exponentially to 0 in $(0,\infty)$ as $m\rightarrow\infty$, we can see that
\begin{align*}
\frac{df}{dt}(t) =&\ \sum_{m=1}^\infty \Big(\lambda_m\phi_{m,*}(t)+\gamma\phi_{m,*}(t) - \gamma\frac{\phi_{m,*}(t)}{\|\Phi_*(t)\|}\Big)\frac{d\phi_{m,*}}{dt}(t)\\
=&\ -\sum_{m=1}^\infty \Big|\frac{d\phi_{m,*}}{dt}(t)\Big|^2.
\end{align*}
Therefore, $f$ is non-increasing and bounded below, i.e., $a=\lim_{t\rightarrow\infty}f(t)$ exists. From the fact that for all $N\in\N$,
$$
\|\Phi_N(t)\|\rightarrow \frac{\gamma}{\gamma+\lambda_1}\ \text{ as } t\rightarrow \infty,
$$
together with \eqref{mNbound}, we can see that
$$
\|\Phi_*(t)\| \rightarrow \frac{\gamma}{\gamma+\lambda_1}\ \text{ as }\ t\rightarrow\infty.
$$
which eventually implies that for each $m\in\N$, $|\frac{d\phi_{m,*}}{dt}(t)|\rightarrow 0$ as $t\rightarrow\infty$, i.e.,
\begin{equation}\label{exactsolutionlimit}
\lim_{t\rightarrow\infty}\Big((\lambda_m+\gamma)-\frac{\gamma}{\|\Phi_*(t)\|}\Big)\phi_{m,*}(t)=0.
\end{equation}
Moreover, since we have $|\phi_{m,*}(t)|\leq \max_{k=1,\dots,k_0}\{M_1\zeta_ke^{-(\lambda_2-\lambda_1)t}\}$ for all $m\geq 2$ with $\zeta_1=1$ and $$\inf_{t\geq 0}\|\Phi_*(t)\|  \geq M_2,$$ we finally obtain, together with \eqref{exactsolutionlimit},
$$
\lim_{t\rightarrow\infty}\|\Phi_*(t)\| =\frac{\gamma}{\lambda_1+\gamma},
$$
which implies that $\lim_{t\rightarrow\infty} |\phi_{1,*}(t)|= \frac{\gamma}{\lambda_1+\gamma},$ and $$u_*(t) \rightarrow \frac{\gamma}{\lambda_1+\gamma}\vec{v} \text{ in } L^2(\Omega)\ \text{ as }\ t\rightarrow\infty$$ for some $\vec{v}\in\{\pm\psi_1\}$.
\qed\end{proof}

\section{Numerical experiments}

 We will now present a few numerical experiments to confirm those analyzed properties of our proposed framework.
 
 \subsection{The gradient descent method}
 In Figure~\ref{fig1},~\ref{fig2},~\ref{fig3} and ~\ref{fig4}, we show the first 25 eigenfunctions of the Laplace operator $-\Delta$ corresponding to their eigenvalues in the increasing order on various domains that are subsets of $(-1,1)\times(-1,1)$ in $\R^2$.
We solved \eqref{model:pde} using the FDM (Finite Difference Method) on uniform grids of size $81\times 81$ including the boundaries, i.e., we represented $(-1,1)\times(-1,1)$ as a uniform grid of size $81\times 81$ and set the interior of each domain to have value 1 and the complement of the interior to  have value 0. As was seen that the solution $u(\cdot, t)$ to \eqref{model:pde} converges to an eigenfunction as $t\rightarrow\infty$, corresponding to the smallest eigenvalue, no eigenvalue estimate was necessary. We set $$dt = t_{n+1}-t_n = 0.17*h^2\ \text{ and }\ dx=dy=h = 0.025=\frac{2}{80},$$ and used
$$
\Big\|\Delta u(t_n) - \gamma\Big(1-\frac{1}{\|u(t_n)\|_2}\Big) u(t_n)\Big\|_2 < 10^{-13}
$$
 as a stopping criterion for all the numerical simulations. When the algorithm stops, $\gamma(\frac{1}{\|u(t_n)\|_2}-1)$ is set to be the eigenvalue corresponding to the eigenfunction $u(\cdot,t_n)$, Using MATLAB, at the $n^{th}$ iteration, representing $u(t_n)$ as an $81\times 81$ matrix, $u(t_{n+1})$ is computed by
 \begin{equation}\label{matlabcode}
 u(t_{n+1}) = u(t_n) + dt*(\Delta u(t_n) - \gamma\Big(1-\frac{1}{\|u(t_n)\|_2}\Big) u(t_n)),
 \end{equation}
 where 
 \begin{align*}
 \Delta u(t_n) =&\ \frac{1}{4}\text{circshift}(u(t_n), [0,\ 1]) + \frac{1}{4}\text{circshift}(u(t_n), [0,\ -1]) \\
 &\ + \frac{1}{4}\text{circshift}(u(t_n), [1,\ 0]) + \frac{1}{4}\text{circshift}(u(t_n), [-1,\ 0])- u(t_n).
 \end{align*}
 When considering various domains $\Omega$, we create a mask, which is an $81\times 81$ matrix representing $\Omega$, and multiply $u(t_n), u(t_{n+1})$ and $\Delta u(t_n)$ by the mask in \eqref{matlabcode}, which makes this method extremely useful since we can deal with various domains by defining the mask matrix only.
\begin{figure}[!htb]
\begin{center}
\includegraphics[scale=0.35]{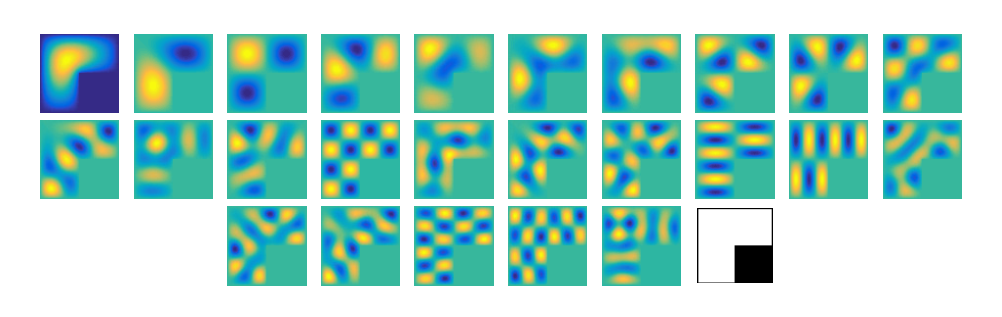} 
\end{center}
\caption{The first 25 eigenfunctions of the Laplace operator on an L shape domain, which is a union of three unit squares in $\R^2$, i.e., $(-1,1)\times(-1,1)\setminus[0,1)\times(-1,0] $, corresponding to their eigenvalues in the increasing order. The last figure visualizes the domain, a uniform grid of size $81\times 81$ including the boundary. The 8th and 9th eigenfunctions correspond to the same eigenvalue 49.2618. The 18th and 19th eigenfunctions correspond to the same eigenvalue 98.2808, The 23rd and 24th eigenfunctions correspond to the same eigenvalue 127.8136.}
\label{fig1}
\end{figure}

To have rough estimates of eigenvalues and eigenfunctions, we can solve \eqref{model:pde} on coarser grids. For example, since the L shape domain in Fig.~\ref{fig1} has a simple structure and there are many previous works (e.g. \cite{N14}, \cite{N16}, \cite{N15}) on computing eigenvalues and eigenvectors of the L shape domain in the literature, we indeed observed that \eqref{model:pde} on a uniform grid of size $21\times21$ computed rough estimates of eigenvalues and eigenfunctions fast.

One thing that we would like to emphasize in Fig.~\ref{fig1} is that our method provides as accurate results as possible on uniform grids. To see this, we would like to take a closer look at the 3rd and 14th eigenfunctions and corresponding eigenvalues. It is known that $u(x,y) =\sin(\pi x)\sin(\pi y)$ for $(x,y)\in (-1,1)\times(-1,1)\setminus[0,1)\times(-1,0]$ is an eigenfunction corresponding to the eigenvalue $2\pi^2$. If we discretize this eigenfunction on a uniform grid of size $81\times 81$ to have $u_{3}^{d}(x_i,y_j) = \sin(\pi x_i)\sin(\pi y_j)$ and compare it with the 3rd eigenfunction $u_3$ that we computed in Fig.~\ref{fig1}, then we have
$$
\|\tilde{u}_3^d - \tilde{u}_3\|_\infty =\max_{1\leq i,j\leq 81} \Big|\tilde{u}_3^{d}(x_i,y_j) - \tilde{u}_3(x_i,y_j)\Big| \sim 10^{-14},
$$
where $\tilde{u}_3^d, \tilde{u}_3$ are the $L^2$ normalizations of $u_3^d, u_3$, and 
$$
|\sum_{i,j}(-\Delta_d\tilde{u}_3^d)(x_i,y_j)\tilde{u}_3^d(x_i,y_j) - \lambda_3| \sim10^{-14},
$$
where $-\Delta_d$ is the discrete Laplace operator on the uniform grid and $\lambda_3$ is the computed eigenvalue $\lambda_3 = \gamma(\frac{1}{\|u_3\|}-1)$. The same is observed for the 14th eigenfunction, as well. As for the 8th and 9th eigenfunctions, we know that the corresponding eigenspace is spanned by the two simpler eigenfunctions $\sin(\pi x)\sin(2\pi y)$ and $\sin(2\pi x)\sin(\pi y)$ than the 8th and 9th eigenfunctions $u_8, u_9$ in Fig.~\ref{fig1}. The true eigenvalue is $5\pi^2$. By setting $\tilde{u}_8^d, \tilde{u}_9^d$ to be the discretized and $L^2$ normalized eigenfunctions from the true simpler ones, we computed the projections $v_8, v_9$ of $u_8,u_9$ onto the eigenspace spanned by $\{\tilde{u}_8^d, \tilde{u}_9^d\}$ and computed $\|u_8-v_8\|_\infty$ and $\|u_9-v_9\|_\infty$ and observed that the $L^\infty$ norms are about $10^{-14}$ implying that $\{u_8, u_9\}$ is indeed a basis for the same eigenspace. After normalizing the projections $v_8, v_9$ by their $L^2$ norms,  we also observed
$$
|\sum_{i,j}(-\Delta_d \tilde{v}_8)(x_i,y_j)\tilde{v}_8(x_i,y_j) - \lambda_8| \sim10^{-14},
$$
where $\lambda_8$ is the computed eigenvalue $\gamma(\frac{1}{\|u_8\|}-1)$, and same for $\tilde{v}_9$. However, when we compare computed eigenfunctions with true eigenfunctions of the form $\sin(n\pi x)\sin(m\pi y)$ on the fixed uniform grid of size $81\times 81$, we observed that the accuracy of computed eigenvalues is still preserved as we increase $n,m$, but the accuracy of computed eigenfunctions deteriorates.

\begin{figure}[!htb]
\begin{center}
\includegraphics[scale=0.35]{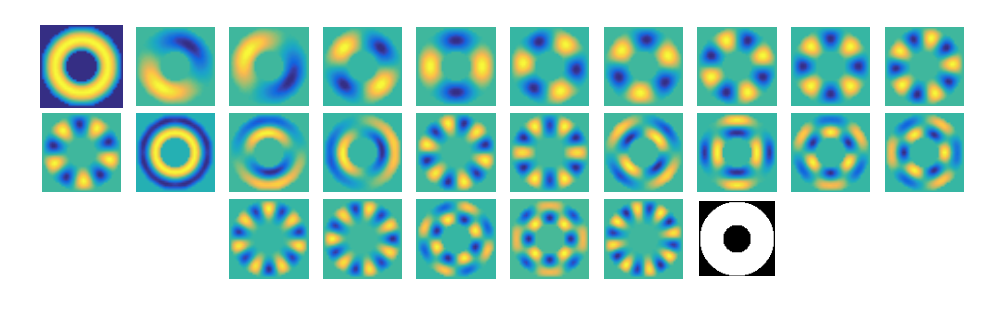} 
\end{center}
\caption{The first 25 eigenfunctions of the Laplace operator on an annulus domain corresponding to their eigenvalues in the increasing order. The last figure visualizes the domain. It's clear which eigenfunctions should correspond to the same eigenvalues.}
\label{fig2}
\end{figure}

As was shown in Fig.~\ref{fig2}, orthogonal eigenfunctions corresponding to the same eigenvalues that are rotations of each other can be expected when the domain is rotationally invariant. However, due to the uniform grid that we use and depending on the rotated angle, we observed that estimated eigenvalues in Fig.~\ref{fig2} can be slightly different for such eigenfunctions unlike the L shape domain in Fig.~\ref{fig1}.

In Fig.~\ref{fig3} and Fig.~\ref{fig4}, we computed eigenfunctions of the Laplace operator on other domains. Especially, we created an domain with no symmetry in Fig.~\ref{fig4}.

\begin{figure}[!htb]
\begin{center}
\includegraphics[scale=0.35]{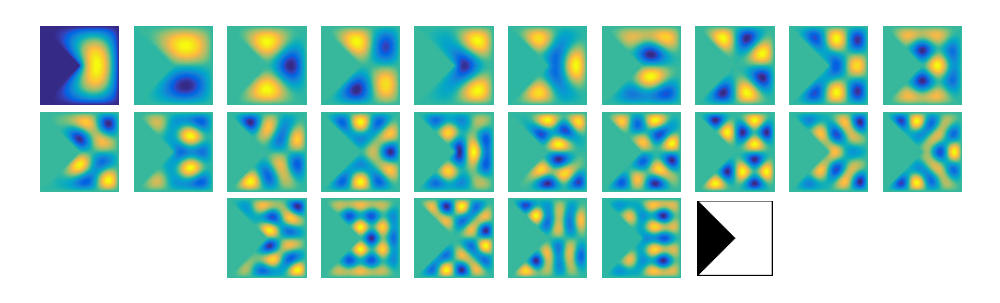} 
\end{center}
\caption{The first 25 eigenfunctions of the Laplace operator on a domain, visualized in the last, corresponding to their eigenvalues in the increasing order.}
\label{fig3}
\end{figure}

\begin{figure}[!htb]
\begin{center}
\includegraphics[scale=0.35]{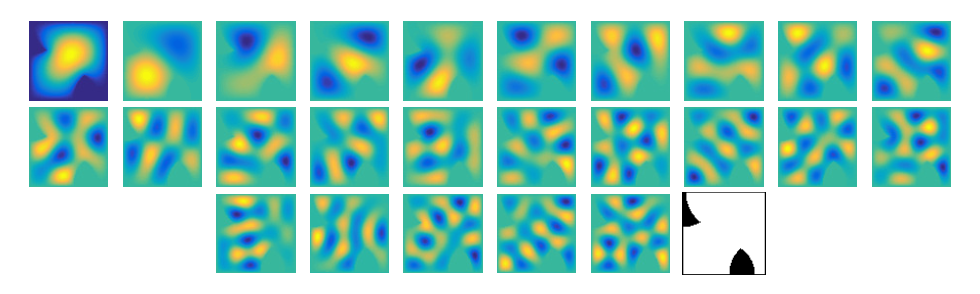} 
\end{center}
\caption{The first 25 eigenfunctions of the Laplace operator on a domain, visualized in the last, corresponding to their eigenvalues in the increasing order.}
\label{fig4}
\end{figure}

Since we provided a generalized eigenvalue problem in the previous work \cite{YK} as an application, we can also see how to apply the Newton's method \eqref{newton1} to solve $Ax = \lambda Bx$ by minimizing $F_{A,B}$, where $A,B$ are symmetric and $B$ is positive definite and $F_{A,B}$ is given by
$$
F_{A,B}(x)= \frac{1}{2}\langle x,Ax\rangle + \frac{\gamma}{2} \langle x,Bx\rangle -\gamma \sqrt{\langle x,Bx\rangle},
$$
with some $\gamma>0$ making $A+\gamma B$ positive definite. We note that minimizing $F_{A,B}$ by the Newton's method with $\|x\|_{B}:=\sqrt{x^TBx}$ generates a sequence $\{x_k\}$ satisfying
\begin{equation}\label{gep:newton}
\Big[\frac{1}{\gamma}A + \Big(1-\frac{1}{\|x_k\|_B}\Big)B + \frac{1}{\|x_k\|_B}\Big(\frac{Bx_k}{\|x_k\|_B}\Big)\Big(\frac{Bx_k}{\|x_k\|_B}\Big)^T\Big]x_{k+1} = \frac{Bx_k}{\|x_k\|_B},
\end{equation}
which can be rewritten as
\begin{equation}\label{newtonRay}
\Big[(A-\lambda_k B) + (\gamma+\lambda_k)\Big(\frac{Bx_k}{\|x_k\|_B}\Big)\Big(\frac{Bx_k}{\|x_k\|_B}\Big)^T\Big]x_{k+1} = \gamma\frac{ Bx_k}{\|x_k\|_B}
\end{equation}
with $\lambda_k = \gamma(\frac{1}{\|x_k\|_B}-1)$. Then,
depending on how to update $\lambda_k$ in \eqref{newtonRay}, either by $\lambda_k = \gamma(\frac{1}{\|x_k\|_B}-1)$ or by $\lambda_k = \frac{x_k^TAx_k}{x_k^TBx_k}$, we end up with either \eqref{gep:newton} or the type of the Rayleigh Quotient Iteration.

However, when we tested the two different eigenvalue update rules above with randomly selected positive definite matrices $A,B$, we observed numerically  that the update rule $\lambda = \gamma(\frac{1}{\|x\|_B}-1)$ found small eigenvalues much more often than the update rule $\lambda = \frac{x^TAx}{x^TBx}$. In fact, we noticed that $\lambda = \gamma(\frac{1}{\|x\|_B}-1)$ found the true smallest eigenvalues much more often than $\lambda=\frac{x^TAx}{x^TBx}$ as can be seen in Fig.~\ref{fig5}, which is interesting to confirm that $\lambda =\gamma(\frac{1}{\|x\|_B}-1)$ is from minimizing the functional $F_{A,B}$ trying to find the smallest eigenvalue.

\begin{figure}[!htb]
\begin{center}
\includegraphics[scale=0.33]{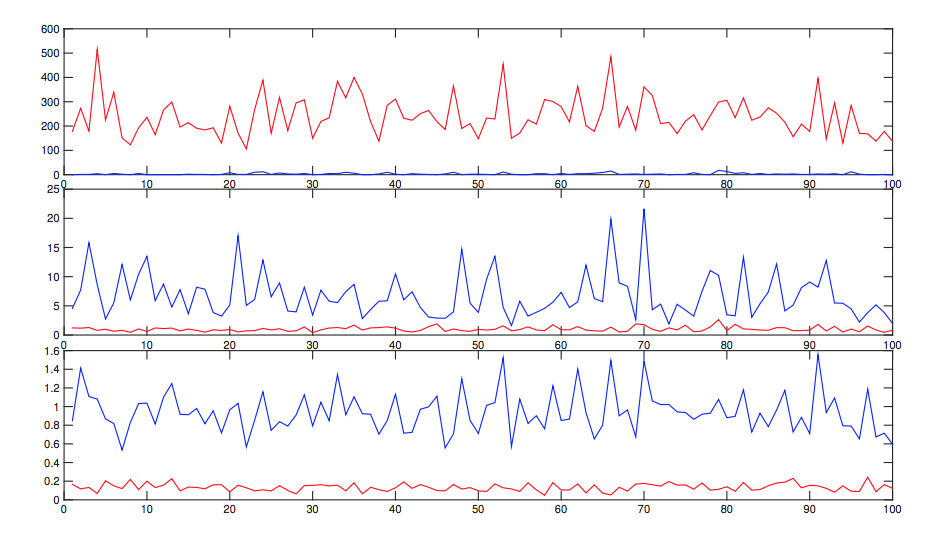}
\end{center}
\caption{100 trials with random positive definite pairs $(A_i,B_i)$, $i=1,2,\dots,100$. $A_i$'s and $B_i$'s are of size $10\times 10$. For each pair $(A_i,B_i)$, we test the two eigenvalue update rules 1000 times starting from random vectors. The {\it red} and {\it blue} colors indicate the two update rules: $\lambda = \gamma(\frac{1}{\|x\|_B}-1)$ {\it in red} and $\lambda = \frac{x^TAx}{x^TBx}$ {\it in blue}. {\bf Top Row}: For each $i=1,\dots,100$, we plot the number of times each update rule finds the true smallest eigenvalue $\lambda_{min}$. We counted the number of eigenvalues computed whose difference from $\lambda_{min}$ is less than $10^{-13}$. {\bf Middle Row}: For each $i=1,\dots,100$, we plot the maximum eigenvalue among the 1000 times trials. {\bf Bottom Row}: For each $i=1,\dots,100$, we plot the mean eigenvalue among the 1000 times trials.}
\label{fig5}
\end{figure}

We also performed the same experiment as in Fig.~\ref{fig5} with different sizes, i.e., $A_i$'s and $B_i$'s are of size $50\times 50$ shown in blue and of size $100\times 100$ shown in red in Fig.~\ref{fig6}. Solid lines indicate results using the update rule $\lambda = \gamma(\frac{1}{\|x\|_B}-1)$ and lines with circular dots indicate results using the other update rule $\lambda = \frac{x^TAx}{x^TBx}$. Fig.~\ref{fig6} clearly shows that the updated rule $\lambda = \gamma(\frac{1}{\|x\|_B}-1)$ tends to find smaller eigenvalues than the update rule $\lambda = \frac{x^TAx}{x^TBx}$. It is interesting to observe that the update rule $\lambda =\frac{x^TAx}{x^TBx}$ never found the smallest eigenvalues. 

\begin{figure}[!htb]
\begin{center}
\includegraphics[scale=0.33]{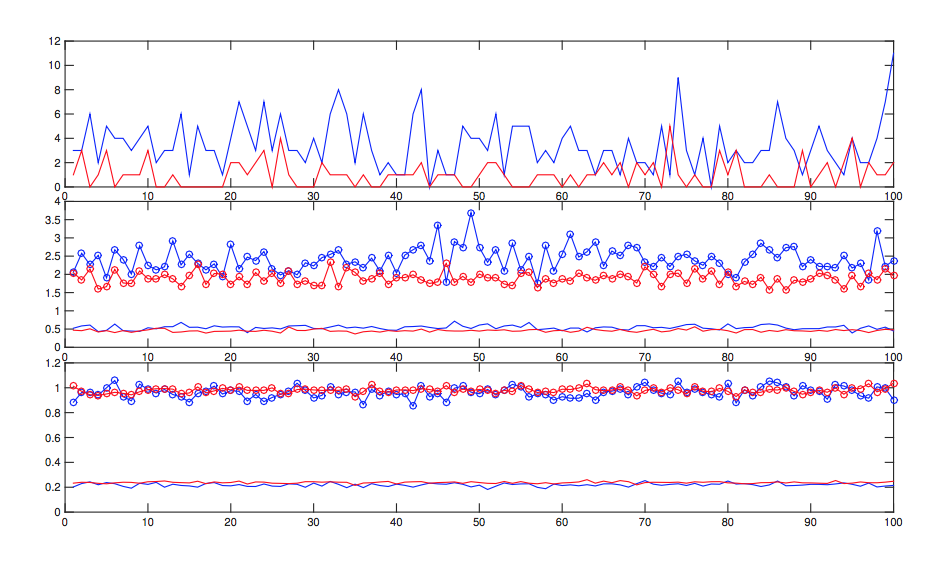}
\end{center}
\caption{The same experiments as in Fig.~\ref{fig5} with different matrix sizes. The red and blue colors indicate the two different sizes: $50\times50$ in blue and $100\times100$ in red. {\it Solid lines} are from the update rule $\lambda = \gamma(\frac{1}{\|x\|_B}-1)$ and the {\it lines with circular dots} are from the update rule $\lambda=\frac{x^TAx}{x^TBx}$. {\bf Top Row}: We counted the number of times that the true smallest eigenvalues were computed among the 1000 times trials for each random positive definite pair $(A_i,B_i)$, $i=1,\dots,100$. Only the counts using the update rule $\lambda = \gamma(\frac{1}{\|x\|_B}-1)$ are shown because the update rule $\lambda=\frac{x^TAx}{x^TBx}$ never found the smallest eigenvalues. {\bf Middle Row}: The maximum eigenvalue among the 1000 times trials for each $(A_i,B_i)$.  {\bf Bottom Row}: The mean eigenvalue among the 1000 times trials for each $(A_i,B_i)$.}
\label{fig6}
\end{figure}

Secondly, we will perform the same numerical experiments as was done with the gradient descent method: finding eigenfunctions of the Laplace operator on various domains. Especially, we will show a comparison result with the L shape experiment. When computing a few eigenfunctions corresponding to their eigenvalues in the increasing order, a drawback of the gradient descent method, besides its rate of convergence, is that one needs to apply any type of the Gram-Schmidt process at every iteration to make sure that the next eigenfunction being searched is orthogonal to the previous ones. This slows down the whole process of computing the first $n$ eigenfunctions quite a bit as $n$ increases. On the other hand, if an estimated eigenvalue $\lambda_k = \gamma(\frac{1}{\|u_k\|}-1)$ at the point $u_k$ is close enough to a true eigenvalue, then the Newton's method would generate a convergent sequence to a nearby critical point, i.e., a corresponding eigenfunction, which implies that it may be unnecessary to apply the Gram-Schmidt process. Hence, when computing many eigenfunctions, the Newton's method can speed up the total time spent significantly if we can find good starting points. Without any fast and efficient ideas combined, we would like to compare the total time spent by the gradient descent method with that by the Newton's method. to find the first $n$ eigenfunctions.

We computed the first 100 eigenfunctions of the Laplace operator on the L shape domain $((-1,0)\times(-1,1))\cup ([0,1)\times(0,1))$ on a uniform grid of size $81\times 81$ using the gradient descent method  and the Newton's method, separately. Having computed the first $k\geq 1$ unit eigenfunctions $\varphi_1,\dots, \varphi_k$ an knowing that the gradient descent method provides convergence to a global minimizer, we can find $u_0$ satisfying the constraints $\langle \varphi_j, u_0\rangle = 0$, $j=1,2,\dots,k$ and 
\begin{equation}\label{comparison:stopping}
\Big\|(-\Delta)u_0 + \gamma\Big(1-\frac{1}{\|u_0\|}\Big)u_0\Big\|<\epsilon
\end{equation}
with $\epsilon=0.1$ or $0.01$, in which case $\gamma(\frac{1}{\|u_0\|}-1)$ is close to the $k+1^{st}$ smallest eigenvalue, and consider $u_{0}$ as a good starting point for both the gradient descent method and the Newton's method to find the $k+1^{st}$ eigenfunction for comparison. Then, we measure the time spent for the gradient descent method to converge starting from $u_{0}$ under the orthogonality constraints and also the time spent for the Newton's method to converge starting from the same $u_{0}$ without any constraints.

In Fig.~\ref{fig7}, we present the amount of time elapsed for both cases: the gradient descent method with the orthogonality constraints vs the Newton's method \eqref{newton1} without any  constraints.

\begin{figure}[!htb]
\begin{center}
\includegraphics[scale=0.33]{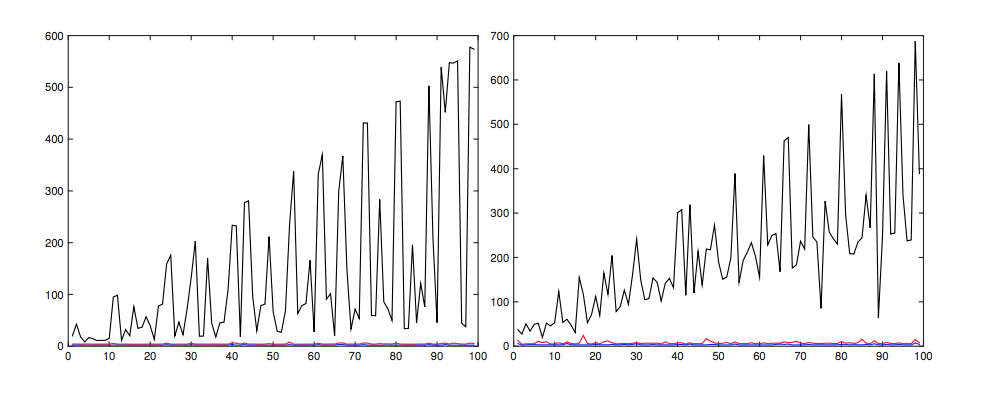}
\end{center}
\caption{Time (sec) elapsed for each eigenfunction computation from a starting point $u_0$ satisfying \eqref{comparison:stopping} with $\epsilon=0.01$ on the left and with $\epsilon=0.1$ on the right, {\bf Black Curve}: the gradient descent method applied to \eqref{eigeq:functional} with the orthogonality constraints, {\bf Red Curve}: the Newton's method \eqref{newton1} without the orthogonality constraints. We can observe a linear increase in time spent due to the linear increase in the number of constraints for the gradient descent method. {\bf Blue Curve}: the Rayleigh quotient iteration.}
\label{fig7}
\end{figure}
We can observe in this experiment that a linear increase in time due to the linear increase in the number of constraints for the gradient descent method is clear and that good starting points even for the gradient method can reduce the computational time significantly.

In Fig.~\ref{fig8}, we compare the Newton's method \eqref{newton1} with the Rayleigh quotient iteration with the same starting points used in Fig.~\ref{fig7}. 
\begin{figure}[!htb]
\begin{center}
\includegraphics[scale=0.4]{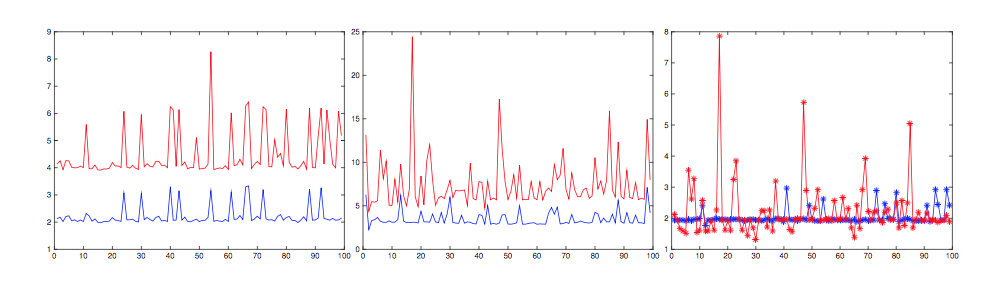}
\end{center}
\caption{Time (sec) elapsed for each eigenfunction computation for the Newton's method \eqref{newton1} and the Rayleigh quotient iteration already shown in Fig.~\ref{fig7}, {\bf Left}: $\epsilon=0.01$ in \eqref{comparison:stopping}, the Newton's method {\it in red} and the Rayleigh quotient iteration {\it in blue}, {\bf Middle}: $\epsilon=0.1$ in \eqref{comparison:stopping}, the Newton's method in red and the Rayleigh quotient iteration in blue, {\bf Right}: the ratio of the time spent for the Newton's method to the time spent for the Rayleigh quotient iteration with $\epsilon=0.01$ in blue and for $\epsilon=0.1$ in red. This numerical experiment shows that the Newton's method \eqref{newton1} is as fast as the Rayleigh quotient iteration.}
\label{fig8}
\end{figure}
As was expected, the gradient descent method with the orthogonality constraints takes much more time to converge, depending on the starting point and on the number of constraints $k\geq1$. On the other hand, the time spent for the Newton's method to compute each eigenfunction depends only on the starting point, showing that almost the same amount of time is required for convergence in computing each eigenfunction. Moreover, the Newton's method \eqref{newton1} presents a simliar rate of convergence as the Rayleigh quotient iteration. In addition, the absence of the orthogonality constraints doesn't allow the Newton's method to find orthogonal eigenfunctions corresponding to the same eigenvalues of multiplicity greater than 1. However, their linear independence is guaranteed by the random sections for the starting points.

In Fig.~\ref{fig9}, we confirm the phenomenon, observed in the example of a generalized eigenvalue problem above, that the eigenvalue update rule $\lambda = \gamma(\frac{1}{\|x\|}-1)$ for the Newton's method \eqref{newton1} finds the smallest eigenvalue much more often than the other update rule $\lambda =\frac{x^TAx}{\|x\|^2}$ for the Rayleigh quotient iteration   with a random starting point.

\begin{figure}[!htb]
\begin{center}
\includegraphics[scale=0.37]{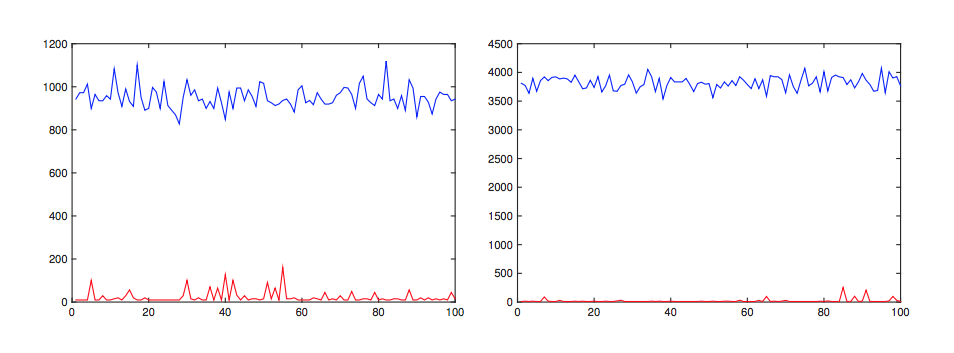}
\end{center}
\caption{Comparison between the Newton's method \eqref{newton1} {\it in red} and the Rayleigh quotient iteration {\it in blue}. We computed eigenvalues of the Laplace operator on the L shape domain by the Newton's method \eqref{newton1} and by the Rayleigh quotient iteration   from the same 100 random starting points. We show the actual eigenvalues computed by the two methods. {\bf Left}: The eigenvalues computed on a uniform grid of size $41\times 41$ are plotted and  53 eigenvalues out of 100 using the Newton's method in red are almost the same as the true smallest eigenvalue. {\bf Right}: The eigenvalues computed on a uniform grid of size $81\times 81$ are plotted and  67 eigenvalues out of 100 using the Newton's method in red are almost the same as the true smallest eigenvalue.}
\label{fig9}
\end{figure}

Two numerical experiments were performed on two different grid sizes, $41\times41$ and $81\times81$, of the L shape domain. We can clearly see that the Newton's method \eqref{newton1} can find the smallest eigenvalue easily, whereas the Rayleigh quotient iteration   can find eigenvalues near the largest one, which shows distinct nature of the two eigenvalue update rules.
 
Before closing this section, we would like to point out that  $$\|Ax_{k+1}- \lambda_kx_{k+1}\|=\Big|\gamma\Big(1-\Big(1+\frac{\lambda_k}{\gamma}\Big)(y_k^Tx_{k+1})\Big)\Big|$$
with $\lambda_k = \gamma(\frac{1}{\|x_k\|}-1)$ for the Newton's method  \eqref{newton1} and with $\lambda_k=\frac{x^TAx}{\|x\|^2}$ for the Rayleigh quotient method   allows for consideration of a stopping criterion
$$
\Big|\gamma\Big(1-\Big(1+\frac{\lambda_k}{\gamma}\Big)(y_k^Tx_{k+1})\Big)\Big|\sim 0,
$$
which is simpler to check than $\|Ax_{k+1}- \lambda_kx_{k+1}\|\sim 0$. We used this stopping criterion for all the experiments.

\section{Conclusion}
\label{conclusion}

In this paper, we proposed and analyzed an unconstrained framework for eigenvalue problems. For practical computation of eigenvectors (or eigenfunctions), we considered the gradient descent method and the Newton's method, and provided proofs for convergence. 

In fact, we showed that applying the gradient descent method guarantees to find a global minimizer of our proposed functional and allows us to find eigenvectors in the increasing order of their corresponding eigenvalues without matrix inversion. We began our discussion with symmetric matrices, however, it turned out that general diagonal matrices can be dealt with similarly. We were also able to provide quantitative analysis on the error in eigenvector estimation. Moreover, the same framework turned out to be applicable not only to finite dimensional cases, but also to infinite dimensional cases. 

It is interesting to note that the functional $F_{A,B}$ in \eqref{eigeq:functional} is nonconvex, yet written in a special form of difference of convex functionals, one of which is quadratically convex, and the other is linearly convex, and is differentiable infinitely many times at every point but zero having the property that local minimizers are global minimizers. This allows us to consider the Newton's method for faster convergence. Indeed, we provided its detailed analysis.

We observed and presented that numerical experiments confirm the theoretical results and our framework provides an easier way of solving eigenvalue problems numerically.

Finally, we would like to point out  that the same framework extends to nonlinear operators as well. One such example is the $p$-Laplacian operator $-\Delta_p$, $1<p<N$: given a bounded domain $\Omega\in\R^N$ with a Lipschitz boundary $\partial\Omega$, find $\varphi$ satisfying
\begin{equation}\label{eigpr:2}
\begin{cases}
-\Delta_p\varphi = \lambda |\varphi|^{p-2}\varphi,&\text{ in }\ \Omega,\\
\quad\ \ \ \varphi = 0,&\text{ on } \partial\Omega,
\end{cases}
\end{equation}
where $-\Delta_p\varphi = \divg(|\nabla \varphi|^{p-2}\nabla \varphi)$ by minimizing the functional $F_p$ defined by
\begin{equation}\label{eigeq:functionalp}
F_p(u) = \frac{1}{p}\int_\Omega |\nabla u(x)|^p dx + \frac{\gamma}{p}\int_\Omega |u(x)|^pdx - \gamma\Big(\int_\Omega |u(x)|^pdx\Big)^{\frac{1}{p}}
\end{equation}

Note that 
$$
\min_{u\in W_0^{1,p}(\Omega)} F_p(u)
$$
exists and a minimizer satisfies \eqref{eigpr:2}. We plan to extend theoretical and numerical aspects of our proposed framework discussed in this paper to applications, especially nonlinear and infinite dimensional cases (e.g. \eqref{eigpr:2}) to reveal what have not been known through conventional methods in finding eigenvalues and eigenfunctions. 

%
%
%
%
%

\section*{acknowledgements}
This work was supported partially by Basic Science Research Program through the National Research Foundation of Korea (NRF) funded by the Ministry of Science, ICT \& Future Planning (NRF-2014R1A1A1002667) and partially by UNIST (1.140074.01). Moreover, I would like to thank my friend and colleague, Ernie Esser, Ph.D., with whom I had fruitful discussions on various topics in image processing. I could not have initiated this project without his inspiration. 



\end{document}